\documentclass[a4paper,11pt]{amsart}
\usepackage[left=2.7cm,right=2.7cm,top=3.5cm,bottom=3cm]{geometry}

\usepackage{amsthm,amssymb,amsmath,amsfonts,mathrsfs,amscd}
\usepackage[latin1]{inputenc}
\usepackage[all]{xy}
\usepackage{latexsym}
\usepackage{longtable}
\usepackage{url}

\newfont{\cyr}{wncyr10 scaled 1100}
\newfont{\cyrr}{wncyr9 scaled 1000}

\theoremstyle{plain}
\newtheorem{theorem}{Theorem}[section]
\newtheorem{proposition}[theorem]{Proposition}
\newtheorem{lemma}[theorem]{Lemma}
\newtheorem{corollary}[theorem]{Corollary}

\theoremstyle{definition}

\newtheorem{definition}[theorem]{Definition}
\newtheorem{assumption}[theorem]{Assumption}

\theoremstyle{remark}
\newtheorem{remark}[theorem]{Remark}

\newcommand{\Q}{\mathbb Q}

\newcommand{\Z}{\mathbb Z}
\newcommand{\R}{\mathbb R}
\newcommand{\C}{\mathbb C}

\DeclareMathOperator{\Aut}{Aut}
\DeclareMathOperator{\Frob}{Frob}

\DeclareMathOperator{\Hom}{Hom}

\DeclareMathOperator{\Gal}{Gal}
\DeclareMathOperator{\GL}{GL}

\DeclareMathOperator{\M}{M}
\DeclareMathOperator{\CH}{CH}

\DeclareMathOperator{\Ta}{Ta}

\newcommand{\res}{\mathrm{res}}
\newcommand{\cores}{\mathrm{cores}}
\newcommand{\cor}{\mathrm{cores}}

\newcommand{\tr}{\mathrm{tr}}
\newcommand{\ord}{\mathrm{ord}}

\newcommand{\cont}{\mathrm{cont}}

\usepackage[usenames]{color}
\definecolor{Indigo}{rgb}{0.2,0.1,0.7}
\definecolor{Violet}{rgb}{0.5,0.1,0.7}
\definecolor{White}{rgb}{1,1,1}
\definecolor{Green}{rgb}{0.1,0.9,0.2}

\newcommand{\longmono}{\mbox{\;$\lhook\joinrel\longrightarrow$\;}}

\newcommand{\longepi}{\mbox{\;$\relbar\joinrel\twoheadrightarrow$\;}}

\newcommand{\dirlim}{\mathop{\varinjlim}\limits}
\newcommand{\invlim}{\mathop{\varprojlim}\limits}

\newcommand{\p}{\mathfrak p}

\newcommand{\cO}{{\mathcal O}}
\newcommand{\E}{{\mathcal E}}
\newcommand{\HH}{{\mathcal H}}

\include{thebibliography}

\begin{document}

\title[Iwasawa theory of Heegner cycles, I]{Iwasawa theory of Heegner cycles, I.\\Rank over the Iwasawa algebra}

%\date{}

\author{Matteo Longo and Stefano Vigni}

\thanks{The two authors are supported by PRIN 2010--11 ``Arithmetic Algebraic Geometry and Number Theory''. The first author is also supported by PRAT 2013 ``Arithmetic of Varieties over Number Fields''; the second author is also supported by PRA 2012 ``Geometria Algebrica e Geometria Aritmetica''.}

\begin{abstract}
Iwasawa theory of Heegner points on abelian varieties of $\GL_2$ type has been studied by, among others, Mazur, Perrin-Riou, Bertolini and Howard. The purpose of this paper, the first in a series of two, is to describe extensions of some of their results in which abelian varieties are replaced by the Galois cohomology of Deligne's $p$-adic representation attached to a modular form $f$ of even weight $>2$. In this more general setting, the role of Heegner points is played by higher-dimensional Heegner cycles in the sense of Nekov\'a\v{r}. In particular, we prove that the Pontryagin dual of a certain Bloch--Kato Selmer group associated with $f$ has rank $1$ over a suitable anticyclotomic Iwasawa algebra.
\end{abstract}

\address{Dipartimento di Matematica, Universit\`a di Padova, Via Trieste 63, 35121 Padova, Italy}
\email{mlongo@math.unipd.it}
\address{Dipartimento di Matematica, Universit\`a di Genova, Via Dodecaneso 35, 16146 Genova, Italy}
\email{vigni@dima.unige.it}



\maketitle

%\tableofcontents

\section{Introduction}

Initiated by Mazur's paper \cite{Maz2}, Iwasawa theory of Heegner points on abelian varieties of $\GL_2$ type (most notably, elliptic curves) has been investigated by, among others, Perrin-Riou (\cite{PR}), Bertolini (\cite{Ber1}, \cite{Ber2}) and Howard (\cite{Ho1}, \cite{Ho2}). A recurrent theme in all these works is the study of pro-$p$-Selmer groups, where $p$ is a prime number, in terms of Iwasawa modules built out of compatible families of Heegner points over the anticyclotomic $\Z_p$-extension of an imaginary quadratic field. In particular, several results on the structure of Selmer groups obtained by Kolyvagin by using his theory of Euler systems (\cite{Kol-Euler}) were generalized to an Iwasawa-theoretic setting.

The goal of the present paper, the first in a series of two, is to describe extensions of some of the results of the previously mentioned authors in which abelian varieties are replaced by the Galois cohomology of Deligne's $p$-adic representation attached to a modular form $f$ of even weight $>2$. In this context, the role of Heegner points is played by Heegner cycles, which were introduced by Nekov\'a\v{r} in \cite{Nek} in order to extend Kolyvagin's theory to Chow groups of Kuga--Sato varieties. More precisely, the main result of our article shows that the Pontryagin dual of a certain Bloch--Kato Selmer group associated with $f$ has rank $1$ over a suitable anticyclotomic Iwasawa algebra. 

Let $N\geq3$ be an integer, let $k\geq4$ be an even integer and let $f$ be a normalized newform of weight $k$ and level $\Gamma_0(N)$, whose $q$-expansion will be denoted by
\[ f(q)=\sum_{n\geq1}a_nq^n. \] 
Fix an imaginary quadratic field $K$ of discriminant coprime to $Np$ in which all the prime factors of $N$ split and let $p$ be a prime number not dividing $N$. For simplicity, we suppose that $\cO_K^\times=\{\pm1\}$, i.e., that $K\not=\Q(\sqrt{-1})$ and $K\not=\Q(\sqrt{-3})$. Fix also embeddings $K\hookrightarrow\C$ and $\bar\Q\hookrightarrow\bar\Q_p$. Write $F$ for the number field generated over $\mathbb Q$ by the Fourier coefficients $a_n$ of $f$ and let $\cO_F$ be its ring of integers. Let $\p$ be a prime ideal of $\mathcal O_F$ above $p$ and denote by $V_{f,\p}$ the $p$-adic representation of $\Gal(\bar\Q/\Q)$ attached to $f$ and $\p$ by Deligne (\cite{Del-Bourbaki}). If $F_\p$ is the completion of $F$ at $\p$ then $V_{f,\p}$ is an $F_\p$-vector space of dimension $2$ equipped with a continuous action of $\Gal(\bar\Q/\Q)$. Write $\cO_\p$ for the valuation ring of $F_\p$ and set $a_p':=a_p\big/p^{k/2-1}$. 

Throughout this article we shall always work under the following

\begin{assumption} \label{ass} 
\begin{enumerate} 
\item $V_{f,\p}$ is non-exceptional (\cite[Definition 6.1]{BesDM});
\item $p\nmid 6N\phi(N)h_K(k-2)!$ where $\phi$ is Euler's function and $h_K$ is the class number of $K$;  
\item  $p$ does not ramify in $F$;
\item \label{ass-cond} $a_p'\in\cO_\p^\times$;
\item $a_p'\not\equiv2\pmod{p}$ if $p$ splits in $K$ and $a_p'\not\equiv1\pmod{p}$ if $p$ is inert in $K$. 
\end{enumerate} 
\end{assumption}

Additional cohomological conditions that allow us to obtain an analogue of Mazur's control theorem for elliptic curves will be listed and made precise in Assumption \ref{ass1}.

In light of results of Serre on eigenvalues of Hecke operators (\cite[\S 7.2]{Serre-Cheb}), it seems reasonable to expect that \eqref{ass-cond} in Assumption \ref{ass} holds for infinitely many primes $p$ (at least if $F\not=\Q$). In fact, questions of this sort appear to lie in the circle of ideas of the conjectures of Lang and Trotter on the distribution of traces of Frobenius automorphisms acting on elliptic curves (\cite{LT}) and of their extensions to higher weight modular forms (\cite{MM}, \cite{MMS}).

Let $K_\infty$ be the anticyclotomic $\Z_p$-extension of $K$ (i.e., the $\Z_p$-extension of $K$ that is dihedral over $\Q$), set $G_\infty:=\Gal(K_\infty/K)$ and form the Iwasawa algebra $\Lambda:=\cO_\p[\![G_\infty]\!]$. As in \cite{Nek}, one can define a $\Gal(\bar{\mathbb Q}/\mathbb Q)$-representation $A$ that is a quotient of the $k/2$-twist of $V_{f,\p}$; denote by $H^1_f(K_\infty,A)$ the Bloch--Kato Selmer group  of $A$ over $K_\infty$. Finally, let \[\mathcal X_\infty:=\Hom_{\cO_\p}^{\rm cont}\bigl(H^1_f(K_\infty,{A}),F_\p/\cO_\p\bigr)\] be the Pontryagin dual of $H^1_f(K_\infty,A)$, which turns out to be finitely generated over $\Lambda$.

Our main result is the following

\begin{theorem} \label{main-thm}
The $\Lambda$-module $\mathcal X_\infty$ has rank $1$.
\end{theorem}

This is the counterpart of \cite[Theorem A]{Ber1}. The key tool in the proof of Theorem \ref{main-thm} is the Iwasawa module of Heegner cycles, which is denoted by $\mathcal H_\infty$ in the main body of the text. The $\Lambda$-module $\mathcal H_\infty$ is built out of the systematic supply of Heegner cycles introduced by Nekov\'a\v{r} (\cite{Nek}), which are higher-dimensional analogues of classical Heegner points. Since we are interested in Iwasawa-theoretic results, we work with Heegner cycles of $p$-power conductor; these cycles live in the Selmer groups $H^1_f(K_m,A)$, where for every $m\geq1$ the field $K_m$ is the subextension of $K_\infty/K$ such that $\Gal(K_m/K)\simeq\Z/p^m\Z$. 

As will become apparent later, our strategy follows \cite{Ber1} closely; in fact, we extend to Heegner cycles the $\Lambda$-adic Kolyvagin method developed by Bertolini. From a slightly different point of view, analogous results could presumably be obtained by exploiting the theory of Kolyvagin systems due to Mazur and Rubin (\cite{MR}). More precisely, we expect that our constructions can easily be adapted to the formalism of \cite{MR}, as was done by Howard in \cite{Ho1} for Heegner points on elliptic curves and by Fouquet in \cite{Fouquet} in the context of big Heegner points, thus leading to one divisibility in the relevant Iwasawa-type Main Conjecture. However, in this paper we preferred to adopt the more direct and explicit approach of Bertolini, and hence to work over the Iwasawa algebra only instead of deducing results for the Iwasawa algebra from specializations to discrete valuation rings as in \cite{MR}.

We observe that a crucial role in our arguments is played by a result of Howard (\cite[Theorem A]{Ho-JNT}) that extends to Heegner cycles and the cohomology of $V_{f,\p}$ a theorem of Cornut (\cite{Co}) asserting the non-triviality of Heegner points on an elliptic curve $E_{/\Q}$ as one ascends $K_\infty$. We also note that a version of Theorem \ref{main-thm} in the context of Hida families of Hilbert modular forms was proved by Nekov\'a\v{r} in \cite[Theorem 12.9.11, (ii)]{Nek-Selmer}. It is worth pointing out that Nekov\'a\v{r}'s result applies to all but finitely many members of a given family and that his proof, relying on his theory of Selmer complexes and (of course) on Hida-theoretic techniques, is different from ours. 

We conclude this introduction by remarking that in the sequel \cite{LV2} to this paper we shall study the torsion submodule $\mathcal X_\infty^{\rm tors}$ of $\mathcal X_\infty$. In particular, we will generalize a divisibility result due to Howard (\cite[Theorem B]{Ho2}) as well as a result of Bertolini concerning the structure of the annihilator of $\mathcal X_\infty^{\rm tors}$ (\cite[Theorem B]{Ber1}). Applications to a Main Conjecture \emph{\`a la} Perrin-Riou will also be given.\\ 

For the convenience of the reader, we collect below all the basic notation that will be used throughout this paper.

\subsubsection*{Notation} \label{notation}

$\bullet$ If $G$ is an abelian group and $n\geq1$ is an integer then $G_n$ denotes the $n$-torsion subgroup of $G$. 

$\bullet$ For any field $K$ we fix an algebraic closure $\bar K$ of $K$ and write $G_K:=\Gal(\bar K/K)$ for the absolute Galois group of $K$. In the special case where $K$ is a number field, we shall view $K$ as a subfield of $\bar{\mathbb Q}$ (in other words, we choose $\bar K=\bar{\mathbb Q}$).

$\bullet$ If $M$ is a continuous $G_K$-module then $H^i_\cont(K,M)$ is the $i$-th continuous cohomology group of $G_K$ with coefficients in $M$. We will usually write $H^i(K,M)$ as a shorthand for $H^i_\cont(K,M)$ and $M(K)$ as a shorthand for $H^0(K,M)$. 

$\bullet$ If $K/F$ is an algebraic field extension then we denote by 
\[ \res_{K/F}:H^i(F,M)\longrightarrow H^i(K,M),\qquad\cores_{K/F}:H^i(K,M)\longrightarrow H^i(F,M) \] 
the restriction and corestriction maps in cohomology, respectively. We recall that, for $K/F$ finite and Galois, we have the equality
\begin{equation} \label{res-cores-norm} 
\res_{K/F}\circ\cores_{K/F}=\tr _{K/F},
\end{equation} 
where $\tr _{K/F}:=\sum_{\sigma\in\Gal(K/F)}\sigma$ is the Galois trace map on $H^i(K,M)$. Further, if $v$ is a place of a number field $K$ then we denote by 
\[ \res_v:H^i(K,M)\longrightarrow H^i(K_v,M) \] 
the restriction (localization) map at $v$, which is determined by the choice of an embedding of $\bar{\mathbb Q}$ into $\bar K_v$.

$\bullet$ If $K$ is a number field and $\ell$ is a prime number then we denote by $K_\ell:=\oplus_{\lambda\mid \ell}K_\lambda$ the direct sum of the completions $K_\lambda$ of $K$ at all the primes $\lambda$ above $\ell$, and we set $H^i(K_\ell,M):=\oplus_{\lambda\mid \ell}H^i(K_\lambda,M)$. 

$\bullet$ If $G$ is a profinite group, $p$ is a prime and $R$ is the ring of integers of a finite extension of $\Q_p$ then $R[\![G]\!]$ is the Iwasawa algebra of $G$ with coefficients in $R$. For an $R[\![G]\!]$-module $M$ we write $M^G$ and $M_G$ for the largest $R$-submodule and quotient of $M$, respectively, on which $G$ acts trivially. 

$\bullet$ If $K$ is a local field then $K^{\rm nr}$ denotes the maximal unramified extension of $K$ inside $\bar K$, so that $I_K:=\Gal(\bar K/K^{\rm nr})$ is the inertia group of $K$. When $K$ is a number field and $v$ is a place of $K$ we also write $I_v$ for $I_{K_v}$. In particular, for all prime numbers $\ell$ we let $F_\ell$ be the arithmetic Frobenius in $\Gal(\Q_\ell^{\rm nr}/\Q_\ell)$. We also fix field embeddings $\bar\Q\hookrightarrow\bar\Q_\ell$ and, with an abuse of notation, when dealing with a $G_\Q$-module that is unramified at $\ell$ we often adopt the same symbol to denote a lift of $F_\ell$ to $G_{\Q_\ell}$ (and its image in $G_\Q$). 

$\bullet$ If $L/E$ is a Galois extension of number fields, $\lambda$ is a prime of $E$ that is unramified in $L$ and $\lambda'$ is a prime of $L$ above $\lambda$ then $\Frob_{\lambda'/\lambda}\in\Gal(L/E)$ denotes the Frobenius substitution at $\lambda'$; the conjugacy class of $\Frob_{\lambda'/\lambda}$ in $\Gal(L/E)$ will be denoted by $\Frob_\lambda$ (notation not reflecting dependence on $L$).

$\bullet$ For an algebraic variety $V$ defined over a field $E$ of characteristic $0$, an integer $r$ such that $0\leq r\leq\dim(V)$ and a prime number $p$, we let 
\[ {\rm AJ}_{V/E}^{(r)}:{\CH}^{r}(V/E)\longrightarrow H^1\bigl(E,H^{2r-1}_{\text{\'et}}(\bar V,\Z_p)\bigr) \] 
denote the corresponding $p$-adic Abel--Jacobi map, with $\bar V:=V\otimes_E\bar E$. We simply write ${\rm AJ}_E$ in place of ${\rm AJ}_{V/E}^{(r)}$ when $V$ and $r$ are clear from the context. 

\section{Bloch--Kato Selmer groups in $\Z_p$-extensions}

Our goal in this section is to introduce the Selmer groups we shall be interested in and state a ``control theorem'' for them.

\subsection{Bloch--Kato Selmer groups} \label{sec-Bloch-Kato} 

We begin with a general discussion of Selmer groups of $p$-adic representations. 

For a number field $E$ and a $p$-adic representation $V$ of the Galois group $G_E$ (by which we mean, as usual, a finite-dimensional $\Q_p$-vector space $V$ equipped with a continuous action of $G_E$), the \emph{Bloch--Kato Selmer group} of $V$ over $E$ (\cite[Sections 3 and 5]{BK}) is the group $H^1_f(E,V)$ that makes the sequence 
\[ 0\longrightarrow H^1_f(E,V)\longrightarrow H^1(E,V)\xrightarrow{\prod_v\partial_v}\prod_vH^1_s(E_v,V) \] 
exact. Here the product is taken over all places of $E$, we set
\[ H^1_s(E_v,V):=H^1(E_v,V)/H^1_f(E_v,V), \] 
the symbol $H^1_f(E_v,V)$ denotes the local Bloch--Kato condition at $v$ (see \cite[\S 2.4]{LV}) and 
\[ \partial_v:H^1(E,V)\longrightarrow H^1_s(E_v,V) \] 
is the composition of the restriction $H^1(E,V)\rightarrow H^1(E_v,V)$ with the canonical projection. If $T$ is a $G_E$-stable lattice in $V$ then set $A:=V/T$ and, for every integer $n\geq1$, let $A_{p^n}$ denote the $p^n$-torsion of $A$. There is a canonical isomorphism $A_{p^n}\simeq T/p^nT$. 

The projection $p:V\twoheadrightarrow A$ and the inclusion $i:T\hookrightarrow V$ induce 
maps 
\[ p: H^1(E,V)\longrightarrow H^1(E,A),\qquad i: H^1(E,T)\longrightarrow H^1(E,V); \] 
let us define $H^1_f(E,A):=p\bigl(H^1_f(E,V)\bigr)$ and $H^1_f(E,T):=i^{-1}\bigl(H^1_f(E,V)\bigr)$. Furthermore, the inclusion $i_n:A_{p^n}\hookrightarrow A$ and the projection $p_n:T\twoheadrightarrow T/p^nT$ induce maps 
\[ i_n:H^1(E,A_{p^n})\longrightarrow H^1(E,A),\qquad p_n:H^1(E,T)\longrightarrow H^1(E,T/p^nT); \]
we set $H^1_f(E,T/p^nT):=p_n\bigl( H^1_f(E,T)\bigr)$ and $H^1_f(E,A_{p^n}):=i_n^{-1}\bigl(H^1_f(E,A)\bigr)$. It can be checked that the isomorphisms $A_{p^n}\simeq T/p^nT$ induce isomorphisms $H^1_f(E,A_{p^n})\simeq H^1_f(E,T/p^nT)$ between Selmer groups.

%These groups sit in the following commutative diagram: 
%\[\xymatrix{
%&&H^1_f(E,A_{p^n})\ar[d]^{i_n}
%\\
%H^1_f(E,T)\ar[r]^{i}\ar@{->>}[d]^{p_n}&
%H^1_f(E,V)\ar@{->>}[r]^{p} &  
%H^1_f(E,A)\\
%H^1_f(E,T/p^nT)&&
%}\] 

Now let $M\in\bigl\{V, T, A, A_{p^n}, T/p^nT\bigr\}$. If $L/E$ is a finite extension of number fields then restriction and corestriction induce maps 
\[ \res_{L/E}:H^1_f(E,M)\longrightarrow H^1_f(L,M),\qquad\cor_{L/E}:H^1_f(L,M)\longrightarrow H^1_f(E,M). \] 
Finally, if $E$ is a number field and $\ell$ is a prime number then we set
\[ H^i_f(E_\ell,M):=\bigoplus_{\lambda|\ell}H^i_f(E_\lambda,M),\quad H^i_s(E_\ell,M):=\bigoplus_{\lambda|\ell}H^i_s(E_\lambda,M),\quad\partial_\ell:=\bigoplus_{\lambda|\ell}\partial_\lambda, \]
the direct sums being taken over the primes $\lambda$ of $E$ above $\ell$.

\subsection{The anticyclotomic $\Z_p$-extension of $K$} \label{anticyclotomic-subsec}

For every integer $m\geq0$ write $H_{p^m}$ for the ring class field of $K$ of conductor $p^m$. By assumption, $p\nmid h_K=|\Gal(H_1/K)|$ and, since $p$ is unramified in $K$, we also have $p\nmid |\Gal(H_p/H_1)|$. Moreover, $\Gal(H_{p^{m+1}}/H_{p})\simeq\Z/p^m\Z$ for all $m\geq1$. It follows that for every $m\geq1$ there is a splitting 
\[ \Gal(H_{p^{m+1}}/K)\simeq G_m\times\Delta \] 
with $G_m\simeq\Z/p^m\Z$ and $\Delta\simeq\Gal(H_p/K)$ of order prime to $p$. For every $m\geq1$ define $K_m$ as the fixed field of $G_m$; then $K_m$ is a subfield of $H_{p^{m+1}}$ such that 
\[ G_m=\Gal(K_m/K)\simeq\Z/p^m\Z. \]
The field $K_\infty:=\cup_{m\geq1}K_m$ is the anticyclotomic $\Z_p$-extension of $K$; equivalently, it is the $\Z_p$-extension of $K$ that is (generalized) dihedral over $\Q$. Set
\[ G_\infty:=\varprojlim_m G_m=\Gal(K_\infty/K)\simeq\Z_p. \] 
For every integer $m\geq0$ define $\Gamma_m:=\Gal(K_\infty/K_m)$. Furthermore, for every $m\geq1$ set $\Lambda_m:=\cO_\p[G_m]$ and define
\[ \Lambda:=\varprojlim_m\Lambda_m=\cO_\p[\![G_\infty]\!]. \]
Here the inverse limit is taken with respect to the maps induced by the natural projections
$G_{m+1}\rightarrow G_m$. For all $m\geq1$ fix a generator $\gamma_m$ of $G_m$ in such a way that ${\gamma_{m+1}|}_{K_m}=\gamma_m$; then $\gamma_\infty:=(\gamma_1,\dots,\gamma_m,\dots)$ is a topological generator of $G_\infty$. It is well known that the map
\begin{equation} \label{iwasawa-isom-eq}
\Lambda\overset\simeq\longrightarrow\cO_\p[\![X]\!],\qquad\gamma_\infty\longmapsto1+X 
\end{equation}
is an isomorphism of $\cO_\p$-algebras (see, e.g., \cite[Proposition 5.3.5]{NSW}).

For an abelian pro-$p$ group $M$ write $M^\vee:=\Hom_{\Z_p}^{\rm cont}(M,\Q_p/\Z_p)$ for its Pontryagin dual, equipped with the compact-open topology (here $\Hom_{\Z_p}^{\rm cont}$ denotes continuous homomorphisms of $\Z_p$-modules and $\Q_p/\Z_p$ is discrete). In the rest of the paper it will be convenient to use also the alternative definition $M^\vee:=\Hom_{\cO_\p}^{\rm cont}(M,F_\p/\cO_\p)$, where $\Hom_{\cO_\p}^{\rm cont}$ stands for continuous homomorphisms of $\cO_\p$-modules and $F_\p/\cO_\p$ is given the discrete topology. It turns out that the two definitions are equivalent, as there is a non-canonical isomorphism between $\Hom_{\cO_\p}^{\rm cont}(M,F_\p/\cO_\p)$ and $\Hom_{\Z_p}^{\rm cont}(M,\Q_p/\Z_p)$ that depends on the choice of a $\Z_p$-basis of $\cO_\p$. See, e.g., \cite[Lemma 2.4]{Bru} for details. 

\subsection{Torsion subgroups in compact modules} \label{torsion-subsec}

In this subsection and the next we collect results on Pontryagin duals that will be used in the sequel and for which we were unable to find convenient references in the literature.

Let $R$ be a topological commutative noetherian ring. In the next sections $R$ will always be a commutative noetherian local ring $(R,\mathfrak m)$ that is an $\mathfrak m$-adically complete $\Z_p$-algebra. Moreover, let $M$ denote a compact Hausdorff abelian group, endowed with a structure of a topological $R$-module. For an ideal $I\subset R$ let $M[I]$ denote the $I$-torsion submodule of $M$, that is 
\[ M[I]:=\bigl\{m\in M\mid\text{$xm=0$ for all $x\in I$}\bigr\}. \] 
Finally, equip the Pontryagin dual $M^\vee$ with the $R$-module structure given by $(x\varphi)(m):=\varphi(xm)$ for all $\varphi\in M^\vee$, $m\in M$ and $x\in R$. 

\begin{lemma} \label{C1} 
There is a canonical isomorphism of $R$-modules 
\[ M^\vee/IM^\vee\simeq M[I]^\vee. \]
\end{lemma}

\begin{proof} Write $I=(x_1,\dots,x_n)$ and consider the map 
\[ \xi:M\longrightarrow\prod_{i=1}^nx_iM,\qquad m\longmapsto{(x_im)}_{i=1,\dots,n}, \] 
whose kernel is equal to $M[I]$. If $i:M[I]\hookrightarrow M$ denotes inclusion then Pontryagin duality gives an exact sequence of $R$-modules
\begin{equation} \label{exact-dual-eq}
\Bigg(\prod_{i=1}^nx_iM\Bigg)^{\!\!\!\vee}\overset{\xi^\vee}\longrightarrow M^\vee\overset {i^\vee} \longrightarrow M[I]^\vee\longrightarrow0; 
\end{equation}
here the surjectivity of $i^\vee$ is a consequence of $M[I]$ being closed in $M$, hence compact. On the other hand, sending $(\varphi_1,\dots,\varphi_n)$ to $\sum_i\varphi_i$ gives an isomorphism between $\prod_i(x_iM)^\vee$ and $\bigl(\prod_ix_iM\bigr)^{\!\vee}$, so we can rewrite \eqref{exact-dual-eq} as
\begin{equation} \label{exact-dual-eq2}
\prod_{i=1}^n(x_iM)^\vee\overset{\xi^\vee}\longrightarrow M^\vee\overset {i^\vee} \longrightarrow M[I]^\vee\longrightarrow0. 
\end{equation}
In light of \eqref{exact-dual-eq2}, we want to check that ${\rm im}(\xi^\vee)=IM^\vee$. First of all, let $\varphi=\sum_{i=1}^nx_i\varphi_i\in IM^\vee$, with $\varphi_i\in M^\vee$ for all $i=1,\dots,n$. Then $\varphi=\xi^\vee\bigl(({\varphi_1|}_{x_1M},\dots,{\varphi_n|}_{x_nM})\bigr)$, which shows that $\varphi\in{\rm im}(\xi^\vee)$. Conversely, let $\varphi\in{\rm im}(\xi^\vee)$; by definition, for every $i=1,\dots,n$ there exists $\varphi_i\in(x_iM)^\vee$ such that $\varphi=\xi^\vee\bigl((\varphi_1,\dots,\varphi_n)\bigr)$. For every $i$, the $R$-module $x_iM$ is compact because $M$ is, hence the inclusion $x_iM\hookrightarrow M$ gives a surjection $M^\vee\twoheadrightarrow(x_iM)^\vee$. Now for every $i=1,\dots,n$ choose a lift $\psi_i\in M^\vee$ of $\varphi_i$. It follows that $\varphi=\sum_{i=1}^nx_i\psi_i\in IM^\vee$, and the proof is complete. \end{proof} 

\subsection{Application to Iwasawa algebras} \label{torsion-iwasawa-subsec}

As in \S \ref{anticyclotomic-subsec}, let $\gamma_\infty$ be a topological generator of $G_\infty\simeq\Z_p$. Let $I_\infty=(\gamma_\infty-1)$ be the augmentation ideal of $\Lambda$ and for every integer $n\geq0$ consider the ideal $I_n:=\bigl(\gamma_\infty^{p^n}-1\bigr)$ of $\Lambda$; in particular, $I_0=I_\infty$. 

Now let $M$ be a continuous $\Lambda$-module. As before, the dual $M^\vee$ inherits a structure of continuous $\Lambda$-module. Since $\gamma_\infty^{p^n}$ is a topological generator of $\Gamma_n$, for all $n\geq0$ there are equalities
\begin{equation} \label{app-pontr}
M[I_n]=M^{\Gamma_n},\qquad M_{\Gamma_n}=M/I_nM.
\end{equation}

\begin{proposition} \label{C2} 
If $M$ is compact then 
\[ {(M^\vee)}_{\Gamma_n}=M^\vee/I_nM^\vee\simeq\bigl(M^{\Gamma_n}\bigr)^\vee \]
for every $n\geq0$.
\end{proposition}

\begin{proof} The equality on the left is just the second equality in \eqref{app-pontr} applied to $M^\vee$, while the canonical isomorphism on the right follows upon taking $I=I_n$ in Lemma \ref{C1} and using the first equality in \eqref{app-pontr}. \end{proof}

\subsection{The Control Theorem} 

Fix a normalized newform $f$ of weight $k$, a prime $\p$ and an imaginary quadratic field $K$ as in Assumption \ref{ass}. 

Write $T$ for the $G_\Q$-representation considered by Nekov\'a\v{r} in \cite[Proposition 3.1]{Nek}, where it is denoted by $A_\p$. This is a free $\cO_\p$-module of rank $2$. The $G_\Q$-representation $V:=T\otimes_{\cO_\p}F_\p$ is then the $k/2$-twist of the representation $V_{f,\p}$. Finally, define the $G_\Q$-representation $A:=V/T$. As above, we shall write $A_{p^n}$ for the $p^n$-torsion submodule of $A$. Observe that
\begin{equation} \label{A-union-eq}
A=\bigcup_{n\geq1}A_{p^n}=\dirlim_nA_{p^n}
\end{equation}
where the direct limit is taken with respect to the natural inclusions $A_{p^n}\hookrightarrow A_{p^{n+1}}$.

\begin{lemma} \label{vanishing-0-lemma}
$H^0(K_m,A)=0$ for all integers $m\geq0$.
\end{lemma}

\begin{proof} Fix an integer $m\geq0$. The extension $K_m/\Q$ is solvable, so $H^0(K_m,A_{p^n})=0$ for all $n\geq0$ by \cite[Lemma 3.10, (2)]{LV}. It follows from \eqref{A-union-eq} that
\[ H^0(K_m,A)=H^0\Big(K_m,\dirlim_nA_{p^n}\Big)=\dirlim_nH^0(K_m,A_{p^n})=0, \]
as was to be shown. \end{proof}

%Since $V_{f,\p}$ is ordinary, it satisfies Panchishkin's condition as a representation of 
%$G_{\Q_p}:=\Gal(\bar\Q_p/\Q_p)$ (cf. \cite[Def. 2.2 and Rem. 2.3]{Ochiai}) 
%and therefore, being a twist of $V_{f,\p}$, the $G_{\Q_p}$-representation $V$ also satisfies Panchishkin's 
%condition (cf. \cite[Rem. 2.3]{Ochiai} again).
%So we deduce the existence of a filtration 
%of $G_{\Q_p}$-modules 
%\[0\longrightarrow T^+\longrightarrow T\longrightarrow T^-\longrightarrow 0.\]
%Further, if 
%$T^*:=\Hom_{\Z_p}(T,\Z_p)$, then $T^*(1)$ also satisfies Panchichkin's condition with filtration 
%\[0\longrightarrow (T^-)^*(1)\longrightarrow T^*(1)\longrightarrow (T^+)^*(1)\longrightarrow 0\]
%where $(T^\pm)^*:=\Hom_{\Z_p}(T^\pm,\Z_p)$ 
%(cf. \cite[Rem. 2.3]{Ochiai} again). 
%Let $V^*:=\Hom_{\Q_p}(V,\Q_p)$ be the $\Q_p$-linear dual of $V$. Then 
%$V^*$ also satisfies Panchichkin's condition (\cite[Rem. 2.3]{Ochiai}). 

To obtain a version of Mazur's Control Theorem (\cite{Maz1}) in our setting, we make the following 

\begin{assumption} \label{ass1}
\begin{enumerate} 
\item For all integers $m\geq1$ and all places $v$ of $K_m$ dividing $N$, the group $H^0(I_{G_{m,v}}, A)$ is $p$-divisible. 
\item For all $m$ and all places $v$ of $K_m$ above $p$, 
%$H^0(G_{K_{m,v}}, (T^+)^*(1))=0$ and $H^0(G_{K_{m,v}},T^-)=0$. 
the canonical projection 
\[ \invlim_n H^1_f\bigl(K_{n,v},T^*(1)\bigr)\longrightarrow H^1_f\bigl(K_{m,v},T^*(1)\bigr) \] 
is surjective, where $T^*:=\Hom_{\Z_p}(T,\Z_p)$ is the $\Z_p$-linear dual of 
$T$ and the inverse limit is taken with respect to the corestriction maps.  
\end{enumerate} 
\end{assumption}

\begin{remark}
Condition (2) in Assumption \ref{ass1} is standard in the case of elliptic curves over $\Q$ (see, e.g., \cite[Assumption 2.15]{BD2} and \cite[\S 1.1]{BD1}) and is satisfied for all but finitely many
primes $p$. Also, conditions similar to (1) in Assumption \ref{ass1} are discussed in \cite[Assumption 2.1]{BD-IMC} in the case of elliptic curves $E$. It would be interesting to investigate analogous results for the representation $V_{f,\p}$. 
\end{remark} 

Define the discrete $\Lambda$-module 
\[ H^1_f(K_\infty,{A}):=\dirlim_mH^1_f(K_m,{A}), \] 
the injective limit being taken with respect to the restriction maps. 

\begin{theorem}[Control Theorem] \label{CT}
For every integer $m\geq0$ there is an isomorphism
\[ \res_{K_\infty/K_m}:H^1_f(K_m,{A})\overset\simeq\longrightarrow H^1_f(K_\infty,{A})^{\Gamma_m}. \] 
\end{theorem}

\begin{proof}[Sketch of proof] This is essentially \cite[Theorem 2.4]{Ochiai}. Fix an integer $m\geq0$. To begin with, the inflation-restriction exact sequence reads
\begin{equation} \label{inf-res-control-eq}
0\longrightarrow H^1\bigl(\Gal(K_\infty/K_m),H^0(K_\infty,A)\bigr)\longrightarrow H^1(K_m,A)\longrightarrow H^1(K_\infty,A)^{\Gamma_m}.
\end{equation}
On the other hand, $H^0(K_m,A)=0$ for all $m$ by Lemma \ref{vanishing-0-lemma}, hence
\[ H^0(K_\infty,A)=\dirlim_mH^0(K_m,A)=0. \]
Thus sequence \eqref{inf-res-control-eq} gives an injection $H^1(K_m,A)\hookrightarrow H^1(K_\infty,A)^{\Gamma_m}$, which in turn restricts to an injection
\[ \res_{K_\infty/K_m}:H^1_f(K_m,{A})\longmono H^1_f(K_\infty,{A})^{\Gamma_m} \]
between Selmer groups. Finally, by comparing with the proof of \cite[Theorem 2.4]{Ochiai} (see, in particular, \cite[p. 81]{Ochiai}), one can check that Assumption \ref{ass1} forces $\res_{K_\infty/K_m}$ to be surjective as well. \end{proof}
 
Let 
\[ \mathcal X_\infty:=\Hom_{\cO_\p}^{\rm cont}\bigl(H^1_f(K_\infty,{A}),F_\p/\cO_\p\bigr) \]
be the Pontryagin dual of $H^1_f(K_\infty,{A})$, equipped with its canonical structure of compact 
$\Lambda$-module. For every integer $m\geq0$ let 
\[ \mathcal X_m:=\Hom_{\mathcal O_\p}^{\rm cont}\bigl(H^1_f(K_m,{A}),F_\p/\mathcal O_\p\bigr) \]
be the Pontryagin dual of $H^1_f(K_m,A)$. Each $\mathcal X_m$ has a natural structure of $\Lambda_m$-module and there is a canonical isomorphism of $\Lambda$-modules $\mathcal X_\infty\simeq\sideset{}{_m}\invlim\mathcal X_m$. Note that, since the Galois representation $A$ is unramified outside $Np$, the $\cO_\p$-modules $\mathcal X_m$ are finitely generated.

\begin{corollary} \label{coro2.2} 
For every $m\geq0$ there is a canonical isomorphism ${(\mathcal X_\infty)}_{\Gamma_m}\simeq \mathcal X_m$. 
\end{corollary} 

\begin{proof} Fix an integer $m\geq0$. Thanks to Proposition \ref{C2}, the isomorphism of Theorem \ref{CT} gives, by duality, a canonical isomorphism 
\begin{equation} \label{coro-CT}
\mathcal X_\infty/I_m\mathcal X_\infty\simeq\mathcal X_m.
\end{equation} 
But, again by Proposition \ref{C2}, the quotient $\mathcal X_\infty/I_m\mathcal X_\infty$ is canonically isomorphic to ${(\mathcal X_\infty)}_{\Gamma_m}$, and we are done. \end{proof} 

\begin{corollary} \label{coro-fin-gen}
The $\Lambda$-module $\mathcal X_\infty$ is finitely generated. 
\end{corollary}

\begin{proof} By choosing $m=0$ in  \eqref{coro-CT}, we obtain an isomorphism $\mathcal X_\infty/I_\infty\mathcal X_\infty\simeq\mathcal X_0$. Since $\mathcal X_0$ is a finitely generated $\mathcal O_\p$-module, the result follows from a topological version of Nakayama's lemma (\cite[Corollary 1.5]{Bru} or \cite[Corollary 5.2.18, (ii)]{NSW}). \end{proof}

\subsection{Projective Selmer modules}

For every integer $m\geq0$ define the $\mathcal O_\p$-module  
\[ S_m:=\invlim_n H^1_f(K_m,A_{p^n}) \]
where the inverse limit is taken with respect to the multiplication-by-$p$ maps. The Tate module of $H^1_f(K_m,A)$ is the $\mathcal O_\p$-module 
\[ \Ta_p\bigl(H^1_f(K_m,A)\bigr):=\invlim_n H^1_f(K_m,A{)}_{p^n} \] 
where $H^1_f(K_m,A{)}_{p^n}$ denotes the $p^n$-torsion submodule of $H^1_f(K_m,A)$ and, again, the inverse limit is taken with respect to the multiplication-by-$p$ maps. 

\begin{lemma} \label{isom-selmer-lemma}
For all $m,n\geq0$ there is an isomorphism 
\[ H^1_f(K_m,A_{p^n})\simeq H^1_f(K_m,A{)}_{p^n}. \]
\end{lemma}

\begin{proof} Fix integers $m,n\geq0$. Taking $G_{K_m}$-cohomology of the short exact sequence 
\[ 0\longrightarrow A_{p^n}\longrightarrow A\overset{p^n}\longrightarrow A\longrightarrow0, \]
where $p^n$ denotes the multiplication-by-$p^n$ map, and using Lemma \ref{vanishing-0-lemma}, we obtain isomorphisms 
\[ H^1(K_m,A_{p^n})\simeq H^1(K_m,A{)}_{p^n} \] 
for all $n\geq0$. These restrict to isomorphisms
\[ H^1_f(K_m,A_{p^n})\simeq H^1_f(K_m,A{)}_{p^n}, \] 
as desired. \end{proof}

\begin{proposition} \label{isom-tate-prop}
For every $m\geq0$ there is a canonical isomorphism 
\[ S_m\simeq\Ta_p\bigl(H^1_f(K_m,A)\bigr). \]
\end{proposition}

\begin{proof} Fix an integer $m\geq0$. Passing to inverse limits over $n$ in Lemma \ref{isom-selmer-lemma} with respect to the multiplication-by-$p$ maps gives the result. \end{proof}

We need one more auxiliary result on Tate modules.

\begin{lemma} \label{tate-isom-lemma}
For every $m\geq0$ there is a canonical isomorphism
\[ \Ta_p\bigl(H^1_f(K_m,A)\bigr)\simeq\Hom_{\cO_\p}(\mathcal X_m,\cO_\p). \] 
\end{lemma}

\begin{proof} Fix an integer $m\geq0$. First of all, $F_\p=\cO_\p[1/p]$ because $p$ is unramified in $F$ by Assumption \ref{ass}, so there is an identification
\begin{equation} \label{tate-hom-eq1}
\Ta_p\bigl(H^1_f(K_m,A)\bigr)=\Hom_{\cO_\p}\bigl(F_\p/\cO_\p,H^1_f(K_m,A)\bigr).
\end{equation} 
On the other hand, taking Pontryagin duals gives a canonical isomorphism
\begin{equation} \label{tate-hom-eq2}
\Hom_{\cO_\p}\bigl(F_\p/\cO_\p,H^1_f(K_m,A)\bigr)\simeq\Hom_{\cO_\p}(\mathcal X_m,\cO_\p).
\end{equation}
The lemma follows by combining \eqref{tate-hom-eq1} and \eqref{tate-hom-eq2}. \end{proof}

Consider the $\Lambda$-modules 
\[ S_\infty:=\invlim_mS_m,\qquad\invlim_m\Ta_p\bigl(H^1_f(K_m,A)\bigr), \] 
the inverse limits being taken with respect to the corestriction maps. There is a sequence of isomorphisms of $\Lambda$-modules 
\[ S_\infty\simeq\invlim_m\Ta_p\bigl(H^1_f(K_m,A)\bigr)\simeq\invlim_m\Hom_{\mathcal O_\p}(\mathcal X_m,\mathcal O_\p)\simeq\invlim_m\Hom_{\mathcal O_\p}\bigl({(\mathcal X_\infty)}_{\Gamma_m},\mathcal O_\p\bigr) \]
where the first isomorphism comes from Proposition \ref{isom-tate-prop}, the second from Lemma \ref{tate-isom-lemma} and the third from Corollary \ref{coro2.2}. Using \cite[Lemma 4, (ii)]{PR}, whose proof can be extended to the case of Iwasawa algebras with coefficients in $\cO_\p$, we obtain canonical isomorphisms  
\begin{equation} \label{selmer-iso1}
S_\infty\simeq\invlim_m\Ta_p\bigl(H^1_f(K_m,A)\bigr)\simeq \Hom_{\Lambda }(\mathcal X_\infty,\Lambda )
\end{equation}
of $\Lambda $-modules. It follows from Corollary \ref{coro-fin-gen} that $S_\infty$ is finitely generated over $\Lambda$. 

\begin{proposition} \label{lemma2.7}
The $\Lambda $-module $S_\infty$ is free of finite rank. 
\end{proposition}

\begin{proof} By \eqref{selmer-iso1}, the $\Lambda$-module ${(S_\infty)}_{G_\infty}$ is isomorphic to $\Hom_{\Lambda }(\mathcal X_\infty,\Lambda )_{G_\infty}$, and this injects into the free $\cO_\p$-module $\Hom_{\cO_\p}\bigl({(\mathcal X_\infty)}_{G_\infty},\cO_\p\bigr)$ thanks to a straightforward extension of \cite[Lemma 4, (iii)]{PR} to our case, where Iwasawa algebras have coefficients in $\cO_\p$. Therefore ${(S_\infty)}_{G_\infty}$ is a free $\cO_\p$-module. Since $S_\infty$ is torsion-free over $\Lambda $ by \eqref{selmer-iso1}, by \cite[Proposition 5.3.19, (ii)]{NSW} (again, generalized to Iwasawa algebras with coefficients in $\cO_\p$) we conclude that $S_\infty$ is free of finite rank over $\Lambda $. \end{proof}

\begin{definition} \label{bloch-kato-dfn}
The \emph{pro-$p$ Bloch--Kato Selmer group of $f$ over $K_\infty$} is the $\Lambda$-module 
\[ \hat H^1_f(K_\infty,{T}):=\invlim_mH^1_f(K_m,{T}), \] 
where the inverse limit is taken with respect to the corestriction maps. 
\end{definition}

The following result is a straightforward consequence of Definition \ref{bloch-kato-dfn}. 

\begin{proposition} \label{prop-selmer}
The $\Lambda$-modules $ \hat H^1_f(K_\infty,{T})$ and $S_\infty$ 
are canonically isomorphic. 
\end{proposition}

\begin{proof} By \cite[p. 261, Corollary]{Tate}, there is a canonical isomorphism 
\[ H^1(K_m,T)\simeq \invlim_nH^1(K_m,T/p^nT) \]
where the inverse limit is taken with respect to the maps induced by the canonical projections $T/p^{n+1}T\twoheadrightarrow T/p^nT$. One easily shows that this isomorphism restricts to an isomorphism between Selmer groups, so we obtain a sequence of canonical isomorphisms 
\[ \hat H^1_f(K_\infty,{T}) \simeq \invlim_m\invlim_n H^1_f(K_m,T/p^nT)\simeq \invlim_m
\invlim_n H^1_f(K_m,A_{p^n})=S_\infty \]
(here recall that the inverse limit over $n$ of the groups  $H^1_f(K_m,A_{p^n})$ is computed with respect to the multiplication-by-$p$ maps, while all the inverse limits over $m$ are taken with respect to the corestriction maps).\end{proof} 

Combining Propositions \ref{lemma2.7} and \ref{prop-selmer}, we obtain  

\begin{corollary} \label{freeness} 
The $\Lambda$-module $\hat H^1_f(K_\infty,{T})$ is free of finite rank. 
\end{corollary}

\section{Iwasawa modules of Heegner cycles} 

\subsection{Preliminaries on Heegner cycles} \label{sec3.1}

We briefly recall construction and basic properties of Heegner cycles, following \cite[Section 5]{Nek} closely. This will also give us the occasion to fix some notation that will be used in the rest of the paper.

Choose an ideal $\mathcal N\subset\cO_K$ such that $\mathcal O_K/\mathcal N\simeq \Z/N\Z$, which exists thanks to the Heegner hypothesis satisfied by $K$. For every integer $c\geq1$ prime to $N$ and the discriminant of $K$, let $\cO_{c}:=\Z+c\cO_K$ be the order of $K$ of conductor $c$. Let $X_0(N)$ be the compact modular curve of level $\Gamma_0(N)$; the isogeny $\C/\cO_{c}\rightarrow \C/(\cO_{c}\cap\mathcal N)^{-1}$ defines a Heegner point $x_c\in X_0(N)$ which, by complex multiplication, is rational over the ring class field $H_{c}$ of $K$ of conductor $c$. Furthermore, let $X(N)$ be the compact modular curve of level $\Gamma(N)$. 

Pick $\tilde x_c\in\kappa^{-1}(x_c)$. The elliptic curve $E_c$ with full level $N$ structure corresponding to $\tilde x_c$ has complex multiplication by $\cO_{c}$. Fix the unique square root 
$\xi_{c}=\sqrt{-Dc}$ of the discriminant of $\cO_{c}$ with positive imaginary part under the chosen embedding of $K$ into $\C$. For all $a\in \mathcal  O_{c}$ write $\Gamma_{c,a}\subset E_c\times E_c$ for the graph of $a$ and $i_{\tilde x_c}:\pi_{k-2}^{-1}(\tilde x_c)\hookrightarrow \tilde{\E}_N^{k-2}$ for the natural inclusion, where $\tilde{\E}_N^{k-2}$ is the Kuga--Sato variety of level $N$ and weight $k$. One can construct two canonical projectors $\Pi_B$ and $\Pi_\epsilon$ (see \cite[Sections 2 and 3]{Nek} for definitions) acting on the Chow groups $\CH^{k/2}(\tilde\E_N^{k-2}/H_c)\otimes\Z_p$. This gives an element   
\begin{equation} \label{cycle-eq1}
\Pi_B\Pi_\epsilon(i_{\tilde x_m})_*\left(\Gamma_{m,\xi_c}^{(k-2)/2}\right)\in\Pi_B\Pi_\epsilon\bigl(\CH^{k/2}(\tilde\E_N^{k-2}/H_{p^m})\otimes\Z_p\bigr). 
\end{equation} 
On the other hand, the Abel--Jacobi map for $\mathcal {\mathcal E}^{k-2}_N$, combined with the projectors $\Pi_B$ and $\Pi_\epsilon$, gives a map 
\begin{equation} \label{AJ}
\Pi_B\Pi_\epsilon\bigl(\CH^{k/2}(\tilde\E_N^{k-2}/H_c)\otimes\Z_p\bigr)\longrightarrow H^1_\cont(H_c,{T})
\end{equation}
(see \cite[Section 4]{Nek} for details). Then one can define the \emph{Heegner cycle} $y_{c}$ in 
$H^1_\cont(H_{c},{T})$ as the image of the cycle \eqref{cycle-eq1} via the Abel--Jacobi map \eqref{AJ}. This class is independent of the choice of $\tilde x_c$ (\cite[p. 107]{Nek}). Finally, since the image of the Abel--Jacobi map is contained in the Bloch--Kato Selmer group, it turns out that 
\[ y_{c}\in H^1_f(H_{c},{T}). \]

\begin{remark}
In \cite{Nek}, the Heegner cycle $y_c$ is introduced only for $c$ square-free and coprime to $NDp$, where $D$ is the discriminant of $K$. However, one can readily check that the construction of $y_c$ carries over without change to our more general setting (in fact, our interest in Iwasawa-theoretic considerations will lead us to specialize to the case where $c$ is a power of $p$).
\end{remark}

\subsection{Hecke action} 

Corestriction from $H_{p^{m+1}}$ to $K_m$ gives an element 
\[ \alpha_m:=\cores_{H_{p^{m+1}}/K_m}(y_{p^{m+1}})\in H^1_f(K_m,{T}). \]  
Put $K_0:=K$ and define
\[ \alpha_0:=\cor_{H_p/K}(y_p)\in H^1_f(K,T). \]
As in the introduction, set $a_p':=a_p\big/p^{k/2-1}$. Recall that, by Assumption \ref{ass}, $a_p'$ belongs to $\cO_\p^\times$ and $a_p'\not\equiv2\pmod{p}$ (respectively, $a_p'\not\equiv1\pmod{p}$) if $p$ splits in $K$ (respectively, if $p$ is inert in $K$).

\begin{lemma} \label{lemma-Heegner} 
If $m\geq2$ then
\[ \cor_{K_{m+1}/K_m}(\alpha_{m+1})=a_p'\alpha_m-\res_{K_m/K_{m-1}}(\alpha_{m-1}), \] 
and 
\[ \cor_{K_1/K}(\alpha_1)=\begin{cases}\bigl(a_p'-a_p'^{-1}(p+1)\bigr)\alpha_0 & \text{if $p$ is inert in $K$,}\\[2mm]\bigl(a_p'-(a_p'-2)^{-1}(p-1)\bigr)\alpha_0 & \text{if $p$ splits in $K$.}\end{cases} \]
\end{lemma}

\begin{proof} Combine the proof of \cite[Proposition 5.4]{Nek} and \cite[Lemma 2, p. 432]{PR}. \end{proof}

\subsection{Projective Heegner modules}  

In order to introduce Iwasawa modules of Heegner cycles, let us first record the following 

\begin{lemma} \label{invertible-lemma}
Let $m\geq1$ be an integer and let $G$ be a non-trivial subgroup of $G_m$. 
For all $u\in\Lambda_m$ the element $a_p-u\sum_{\sigma\in G}\sigma$ is invertible in $\Lambda_m$.
\end{lemma}

\begin{proof} For simplicity, set $\beta:=a_p-u\sum_{\sigma\in G}\sigma$. Suppose that $|G|=p^t$ with $1\leq t\leq m$ and let $b:=p^{m-t}$, so that $G$ is generated by $\gamma_m^b$. Write $\pi_m:\Lambda\rightarrow\Lambda_m$ for the canonical projection, fix a lift $\tilde u\in\Lambda$ of $u$ under $\pi_m$ and let $G(X)\in\cO_\p[\![X]\!]$ be the power series corresponding to $\tilde u$ under isomorphism \eqref{iwasawa-isom-eq}, which we can assume to be a polynomial. The element 
\[ \tilde\beta:=a_p-\tilde u\cdot\sum_{i=1}^{p^t}\gamma_\infty^{bi}\in\Lambda \]
is then a lift of $\beta$ via $\pi_m$ that is sent by isomorphism \eqref{iwasawa-isom-eq} to
\[ F(X):=a_p-G(X)\cdot\sum_{i=1}^{p^t}(1+X)^{bi}=a_p-p^t\eta+Xf(X)\in\cO_\p[\![X]\!] \]
for certain $\eta\in\cO_\p$ and $f(X)\in\cO_\p[X]$. But $a_p-p^t\eta\in\cO_\p^\times$ because $a_p$ is a $p$-adic unit by Assumption \ref{ass}, hence $F(X)$ is invertible in $\cO_\p[\![X]\!]$. It follows that $\tilde\beta$ is invertible in $\Lambda$, and then $\beta=\pi_m(\tilde\beta)$ is invertible in $\Lambda_m$, as required. \end{proof}

For every integer $m\geq0$ denote by $\HH_m$ the $\Lambda_m$-submodule of $H^1_f(K_{m},{T})$ generated by $\alpha_m$. 

\begin{proposition} \label{prop3.3}
Corestriction induces surjective maps 
\[ \cores_{K_{m+1}/K_m}:\HH_{m+1}\longepi\HH_m \] 
for all integers $m\geq0$. 
\end{proposition}

\begin{proof} As in the proof of \cite[Proposition 4]{Ber1}, we proceed by induction on $m$. By Assumption \ref{ass}, the number $a_p'-a_p'^{-1}(p+1)$ (respectively, $a_p'-(a_p'-2)^{-1}(p-1)$) is a $p$-adic unit when $p$ is inert in $K$ (respectively, $p$ splits in $K$), and then the second formula in Lemma \ref{lemma-Heegner} proves the proposition for $m=0$. Now suppose that the claim is true for $m-1$. Then 
\begin{equation}\label{equation9}
\alpha_{m-1}=u\cdot\cor_{K_m/K_{m-1}}(\alpha_m)\end{equation}  for some $u\in\Lambda_{m-1}^\times$,
and combining this equality with the first formula in Lemma \ref{lemma-Heegner} gives
\begin{equation} \label{induction-eq}
\cor_{K_{m+1}/K_m}(\alpha_{m+1})=(a_p'-u\cdot\cor_{K_m/K_{m-1}})(\alpha_m).
\end{equation}
Since $a_p'-u\cdot\cor_{K_m/K_{m-1}}\in\Lambda_m^\times$ by Lemma \ref{invertible-lemma}, the claim for $m$ follows from \eqref{induction-eq}. \end{proof}

\begin{definition}
The \emph{Iwasawa module of Heegner cycles} is the compact $\Lambda$-module
\[ \HH_\infty:=\invlim_m\HH_m\subset\hat H^1_f(K_\infty,{T}), \] 
where the inverse limit is taken with respect to the corestriction maps. 
\end{definition}

The proof of the following result crucially exploits a theorem of Howard (\cite{Ho-JNT}) that extends to the higher weight setting results of Cornut (\cite{Co}) on the generic non-vanishing of Heegner points on elliptic curves over anticyclotomic $\Z_p$-extensions.

\begin{theorem} \label{prop3.5}
The $\Lambda$-module $\HH_\infty$ is free of rank $1$. 
\end{theorem}

\begin{proof} First of all, observe that $\HH_\infty$ is a $\Lambda$-submodule of $\hat H^1_f(K_\infty,T)$, which is free of finite rank over $\Lambda$ by Corollary \ref{freeness}. It follows that $\HH_\infty$, being cyclic, is either trivial or isomorphic to $\Lambda$. On the other hand, \cite[Theorem A]{Ho-JNT} ensures that there is an integer $m\geq0$ such that $\alpha_m$ is not $\cO_\p$-torsion in $H^1_f(K_m,T)$. Set $x_m:=\alpha_m\in\HH_m$ and for every $n\geq m+1$ choose $x_n\in\HH_n$ such that 
\[ \cores_{K_{n}/K_{n-1}}(x_n)=x_{n-1}. \]
Since $x_m$ is non-torsion, the $\mathcal O_\p$-submodule of $\HH_\infty$ generated by this compatible sequence is isomorphic to $\mathcal O_\p$, so $\HH_\infty$ cannot be trivial. \end{proof}

\subsection{Injective Heegner modules} \label{sec-inj-modules}

Now we introduce the $\Lambda$-module $\mathcal E_\infty$, which is obtained by taking an inductive limit of Heegner cycles. First note the following easy consequence of Proposition \ref{prop3.3}. 

\begin{corollary} \label{coro3.4}
Restriction induces injective maps 
\[ \res_{K_{m+1}/K_m}:\HH_m\;\longmono\;\HH_{m+1} \] 
for all integers $m\geq0$. 
\end{corollary} 

\begin{proof} Since $\res_{K_{m+1}/K_m}$ is injective, it suffices to show that its image is contained in $\mathcal H_{m+1}$. By Proposition \ref{prop3.3}, we know that $\alpha_m=u_m\cdot \cores_{K_{m+1}/K_m}(\alpha_{m+1})$ for some $u_m\in \Lambda_m^\times$, and the corollary follows by applying restriction and using \eqref{res-cores-norm}. \end{proof} 

For every integer $m\geq0$ let $\bar\alpha_m$ be the image of $\alpha_m$ in $H^1_f(K_m,T/p^mT)$ via the map induced by the projection $T\twoheadrightarrow T/p^mT$. Denote by $\beta_m$ the image of $\bar\alpha_m$ under the map induced by the isomorphism $T/p^mT\simeq A_{p^m}$, then define $\mathcal E_m:=\Lambda_m\beta_m$ as the $\Lambda_m$-submodule of $H^1_f(K_m,A_{p^m})$ generated by $\beta_m$. By Corollary \ref{coro3.4}, $\res_{K_{m+1}/K_m}$ sends the $\Lambda_m$-module generated by $\bar\alpha_m$ to the $\Lambda_{m+1}$-module generated by $\bar\alpha_{m+1}$. Therefore composing with the map obtained from the natural inclusion $A_{p^m}\subset A_{p^{m+1}}$ gives canonical maps 
$\rho_m:\mathcal E_m\rightarrow \mathcal E_{m+1}$, and we can consider the discrete $\Lambda$-module 
\[ \mathcal E_\infty:=\dirlim_m\mathcal E_m, \]  
the direct limit being taken with respect to the maps $\rho_m$. 

\begin{proposition} \label{inclusion-E-prop}
There is an injection of $\Lambda$-modules $\mathcal E_\infty\hookrightarrow H^1_f(K_\infty,A)$.
\end{proposition}

\begin{proof} By Lemma \ref{isom-selmer-lemma}, there are isomorphisms
\[ H^1_f(K_m,A_{p^m})\simeq H^1_f(K_m,A{)}_{p^m} \]
for all integers $m\geq0$. By definition, every $\mathcal E_m$ is a submodule of $H^1_f(K_m,A_{p^m})$, and then the (left) exactness of the direct limit gives injections
\[ \begin{split}
   \mathcal E_\infty=\dirlim_m\mathcal E_m&\longmono\dirlim_m H^1_f(K_m,A_{p^m})\simeq\dirlim_m H^1_f(K_m,A{)}_{p^m}\\
  &\longmono\dirlim_m H^1_f(K_m,A)=H^1_f(K_\infty,A), 
   \end{split} \]
which completes the proof. \end{proof}

Taking the Pontryagin dual, we get a compact $\Lambda$-module 
\[ \mathcal E_\infty^\vee:=\Hom_{\mathcal O_\p}^{\rm cont}(\mathcal E_\infty,F_\p/\mathcal O_\p)\simeq\invlim_m \mathcal E_m^\vee \] 
where, for each $m$, the module $\mathcal E_m^\vee:=\Hom_{\mathcal O_\p}^{\rm cont}(\mathcal E_m,F_\p/\mathcal O_\p)$ is the Pontryagin dual of $\mathcal E_m$. Proposition \ref{inclusion-E-prop} gives a surjection of $\Lambda$-modules 
\begin{equation} \label{pi}
\pi:\mathcal X_\infty\longepi\mathcal E_\infty^\vee.
\end{equation} 
In particular, there is also an injection of $\Lambda$-modules 
\begin{equation} \label{hom-inj-eq}
\Hom_{\Lambda}\bigl(\mathcal E_\infty^\vee,\Lambda\bigr)\longmono\Hom_{\Lambda}(\mathcal X_\infty,\Lambda)\simeq\hat H^1_f(K_\infty,{T}), 
\end{equation}
where the isomorphism on the right is a consequence of \eqref{selmer-iso1} and Proposition \ref{prop-selmer}. The next result describes the image of this map.

\begin{proposition} \label{hom-image-prop}
The image of \eqref{hom-inj-eq} is equal to $\mathcal H_\infty$.
\end{proposition}

\begin{proof} A more or less tautological, albeit somewhat tedious, diagram chasing. We omit the details. \end{proof}

It follows from Proposition \ref{hom-image-prop} that there is an isomorphism  
\[ \Hom_{\Lambda}\bigl(\mathcal E_\infty^\vee,\Lambda\bigr)\simeq\mathcal H_\infty \] 
of $\Lambda $-modules. In particular, Theorem \ref{prop3.5} implies that $\mathcal E^\vee_\infty$ has rank $1$ over $\Lambda$.

\section{The Euler system argument} \label{euler-sec}

The aim of this section is to prove Theorem \ref{main-thm}; we restate it below.
 
\begin{theorem} \label{thm4.1}
The rank of $\mathcal X_\infty$ over $\Lambda$ is $1$. 
\end{theorem} 

By Theorem \ref{prop3.5}, the $\Lambda$-module $\mathcal H_\infty$ of Heegner cycles is free of rank $1$. Furthermore, the discussion of \S \ref{sec-inj-modules} shows that the $\Lambda$-rank of $\mathcal E_\infty^\vee$ is $1$. Recall that, by Corollary \ref{coro-fin-gen}, the $\Lambda$-module $\mathcal X_\infty$ is finitely generated. The surjection $\pi:\mathcal X_\infty\twoheadrightarrow \mathcal E_\infty^\vee$ of $\Lambda$-modules introduced in \eqref{pi} is $\Gal(K/\Q)$-equivariant. Since $\mathcal E_\infty^\vee$ has rank $1$ over $\Lambda$, proving Theorem \ref{thm4.1} is equivalent to showing that the $\Lambda$-module $\ker(\pi)$ is torsion. Let $\tau$ be the generator of $\Gal(K/\Q)$. In order to prove Theorem \ref{thm4.1} it suffices to show that every $y\in\ker(\pi)$ lying in an eigenspace for $\tau$ (i.e., such that $\tau(y)=\pm y$) is $\Lambda$-torsion.

As in \S \ref{anticyclotomic-subsec}, let $\gamma_\infty$ be a topological generator of $G_\infty$. We need two lemmas.

\begin{lemma} \label{generator-lemma}
There is a generator $\omega$ of the ideal $(\gamma_\infty-1)\Lambda$ such that $\omega^\tau=-\omega$.
\end{lemma}

\begin{proof} Define $\omega:=\gamma_\infty-\gamma_\infty^{-1}\in\Lambda$. Since $\gamma_\infty^\tau=\gamma_\infty^{-1}$, one has $\omega^\tau=-\omega$. On the other hand, one can write $\omega=\gamma_\infty^{-1}(\gamma_\infty+1)(\gamma_\infty-1)$. The isomorphism \eqref{iwasawa-isom-eq} sends $\gamma_\infty+1$ to $2+X$, which is invertible in $\cO_\p[\![X]\!]$ because $p\not=2$, and it follows that $\omega$ is a generator of $(\gamma_\infty-1)\Lambda$. \end{proof}

\begin{lemma} \label{eigen-torsion-lemma}
Let $\epsilon\in\{\pm\}$ and suppose that every $y\in\ker(\pi)^\epsilon$ is $\Lambda$-torsion. Then every $y\in\ker(\pi)$ lying in an eigenspace for $\tau$ is $\Lambda$-torsion.
\end{lemma}

\begin{proof} By Lemma \ref{generator-lemma}, there is a generator $\omega$ of $(\gamma_\infty-1)\Lambda$ such that $\omega^\tau=-\omega$. Now pick $y\in\ker(\pi)^{-\epsilon}$. Then $\omega y\in\ker(\pi)$ and $\tau(\omega y)=\epsilon\omega y$, so $\omega y\in\ker(\pi)^\epsilon$. It follows that $\omega y$ is $\Lambda$-torsion, hence $y$ is $\Lambda$-torsion as well. \end{proof} 

Choose an element $x\in \mathcal X_\infty$ such that $\tau(x)=\epsilon x$ for some $\epsilon\in\{\pm\}$ and $\pi(x)\neq0$. Thanks to Lemma \ref{eigen-torsion-lemma}, in order to prove Theorem \ref{thm4.1} it is enough to show that every $y\in\ker(\pi)^{-\epsilon}$ is $\Lambda$-torsion. To do this, we will adapt the $\Lambda$-adic Euler system argument of \cite{Ber1}, as explained in the next subsections. 

\subsection{Kolyvagin primes} 

Denote by 
\[ \rho_m:G_\Q\longrightarrow \Aut(A_{p^m})\simeq\GL_2(\cO_\p/p^m\cO_\p) \] 
the Galois representation on $A_{p^m}$ and let $K(A_{p^m})$ be the composite of $K$ and the field cut out by $\rho_m$; in other words, $K(A_{p^m})$ is the composite of $K$ and $\bar\Q^{\ker(\rho_m)}$. In particular, $K(A_{p^m})$ is Galois over $\Q$. 

\begin{definition} \label{kolyvagin-dfn}
A prime number $\ell$ is a \emph{Kolyvagin prime} relative to $p^m$ if $\ell\nmid NDp$ and $\Frob_\ell=[\tau]$ in $\Gal(K(A_{p^m})/\Q)$. 
\end{definition} 

In particular, Kolyvagin primes are inert in $K$. Define
\[ R_m:=\cO_\p/p^m\cO_\p[G_m]. \]
Likewise, set $D_m:=\Gal(K_m/\Q)$, 
let $\tilde R_m$ be the quotient of $\Lambda_m$ (and of $\cO_\p[D_m]$) given by 
\[ \tilde R_m:=\cO_\p/p^m\cO_\p[D_m] \]
and, for each choice of sign $\pm$, let $R_m^{(\pm)}$ be the $\tilde R_m$-module $R_m$ with $\tau$ acting on group-like elements by $\gamma^\tau:=\pm \gamma^{-1}$. In particular, $R_m^{(+)}$ corresponds to the linear extension of the natural action of $\tau$ on $G_m$.

Let $\ell$ be a Kolyvagin prime relative to $p^m$ and let $\lambda$ be the unique prime of $K$ 
above $\ell$; let $K_\lambda$ be the completion of $K$ at $\lambda$. Evaluation at Frobenius gives a $\Gal(K/\Q)$-equivariant isomorphism 
\[ H^1_f(K_\lambda,A_{p^m})\simeq A_{p^m} \] 
for every $m\geq1$ (see, e.g., \cite[Lemma 6.8]{BesDM}). Since $\lambda$ splits completely in $K_m$, for each choice of sign $\pm$ we obtain an isomorphism of $\tilde R_m$-modules 
\begin{equation} \label{lemma-kol}
H^1_f(K_{m,\ell},A_{p^m})^{(\pm)}\simeq R_m^{(\pm)}.
\end{equation}
If $\ell$ is a Kolyvagin prime relative to $p^m$ and $M\in\bigl\{V, T, A, A_{p^m}, T/p^mT\bigr\}$ then we let
\[ \res_\ell=\res_{m,\ell}:H^1(K_m,M)\longrightarrow H^1(K_{m,\ell},M) \] 
be the direct sum $\oplus_{\lambda|\ell}\res_\lambda$ of the local restrictions $\res_\lambda$ where the sum ranges over all the primes $\lambda$ of $K_m$ above $\ell$. We will also use the same symbol for the restriction of $\res_\ell$ to subgroups of $H^1(K_m,M)$. 

\subsection{Action of complex conjugation} \label{sec-complex-conj}

In this subsection we study the action of $\Gal(K/\Q)$ on Selmer groups. These results will be used in \S \ref{sec-families} to show the existence of suitable families of Kolyvagin primes. 

Define $\tilde\Lambda:=\cO_\p[\![D_\infty]\!]$ to be the Iwasawa algebra of $D_\infty:=\Gal(K_\infty/\Q)$ with coefficients in $\cO_\p$. As before, for each sign $\pm$ write $\Lambda^{(\pm)}$ for the ring $\Lambda$ viewed as a module over $\tilde\Lambda$ via the action of $\tau$ given by $\gamma^\tau=\pm\gamma^{-1}$ for all $\gamma\in G_\infty$. In particular, $\Lambda^{(+)}$ corresponds to the linear extension of the natural action of $\tau$ on $G_\infty$.

The canonical action of $\tau$ on $\mathcal X_\infty$ makes it into a $\tilde\Lambda$-module. Recall the element $x\in \mathcal X_\infty$ chosen at the beginning of this section such that $\pi(x)\neq0$, where $\pi$ is the map in \eqref{pi}, and $\tau(x)=\epsilon x$ for some $\epsilon\in\{\pm\}$. Now pick an element $y\in\ker(\pi)^{-\epsilon}$ and consider the surjection of $\tilde\Lambda$-modules
\[ \Lambda^{(\epsilon)}\oplus\Lambda^{(-\epsilon)}\longepi\Lambda x\oplus\Lambda y\subset\mathcal X_\infty\oplus\mathcal X_\infty,\qquad(\xi,\eta)\longmapsto(\xi x,\eta y). \] 
Since $\mathcal E_\infty^\vee$ is torsion-free, $\ker(\pi)\cap \Lambda x=\{0\}$, hence $\Lambda x\cap\Lambda y=\{0\}$. Therefore the canonical map of $\tilde\Lambda$-modules $\Lambda x\oplus\Lambda y\rightarrow \mathcal X_\infty$ given by the sum is injective. Composing the last two maps, we get a map of $\tilde\Lambda$-modules 
\begin{equation} \label{theta-eq}
\vartheta:\Lambda^{(\epsilon)}\oplus\Lambda^{(-\epsilon)}\longrightarrow\mathcal X_\infty.
\end{equation}
that sends $(\alpha,\beta)$ to $\alpha x+\beta y$.

\begin{lemma} \label{dual-p^m-lemma}
For every $m\geq1$ there is a canonical isomorphism
\[ \mathcal X_m/p^m\mathcal X_m\simeq H^1_f(K_m,A_{p^m})^\vee. \]
\end{lemma}

\begin{proof} By definition, $\mathcal X_m$ is the Pontryagin dual of $H^1_f(K_m,A)$. Applying Lemma \ref{C1} with $M=H^1_f(K_m,A)$ and $I=(p^m)$, we see that $\mathcal X_m/p^m\mathcal X_m$ is isomorphic to the Pontryagin dual of $H^1_f(K_m,A{)}_{p^m}$. But $H^1_f(K_m,A{)}_{p^m}$ is isomorphic to $H^1_f(K_m,A_{p^m})$ by Lemma \ref{isom-selmer-lemma}, and the claim follows. \end{proof}

For every integer $m\geq1$ consider the surjection  
\begin{equation} \label{p_m-eq}
p_m:\mathcal X_\infty\longepi (\mathcal X_\infty{)}_{\Gamma_m}\simeq\mathcal X_m\longepi\mathcal X_m/p^m\mathcal X_m\simeq H^1_f(K_m,A_{p^m})^\vee, 
\end{equation}
where the first isomorphism comes from Corollary \ref{coro2.2} and the second from Lemma \ref{dual-p^m-lemma}. Let us define the following $\tilde R_m$-submodules of $H^1_f(K_m,A_{p^m})^\vee$: 
\[ \begin{split}
   Z_m&:=\bigl((p_m\circ\vartheta)(\Lambda^{(\epsilon)}\oplus\{0\}))\bigr)\cap\big((p_m\circ\vartheta)(\{0\}\oplus \Lambda^{(-\epsilon)})\big),\\
  W_m^{(\epsilon)}&:=\big((p_m\circ\vartheta)(\Lambda^{(\epsilon)}\oplus\{0\}) \big)/Z_m,\\
  W_m^{(-\epsilon)}&:= \big((p_m\circ\vartheta)(\{0\}\oplus \Lambda^{(-\epsilon)})\big)/Z_m.
   \end{split} \]
Now set
\[ \Sigma_m:=\Big(H^1_f(K_m,A_{p^m})^\vee\big/Z_m\Big)^{\!\vee}, \] 
so that there is a canonical identification $\Sigma_m^\vee=H^1_f(K_m,A_{p^m})^\vee/Z_m$. Since the submodule $Z_m$ is closed in $H^1_f(K_m,A_{p^m})^\vee$, Pontryagin duality applied to $H^1_f(K_m,A_{p^m})$ yields a natural injection $\Sigma_m\hookrightarrow H^1_f(K_m,A_{p^m})$ of $\tilde R_m$-modules, which we shall often view as an inclusion. Therefore we obtain a chain of maps of $\tilde R_m$-modules 
\begin{equation} \label{chain-eq}
\Lambda^{(\epsilon)}\oplus\Lambda^{(-\epsilon)}\longepi W_m^{(\epsilon)}\oplus W_m^{(-\epsilon)}\longmono\Sigma_m^\vee, 
\end{equation}
where the surjection is induced by $p_m\circ\vartheta$ and the injection is given by the sum. By construction, the composition in \eqref{chain-eq} factors through the surjection $\Lambda^{(\epsilon)}\oplus\Lambda^{(-\epsilon)}\twoheadrightarrow R^{(\epsilon)}_m\oplus R^{(-\epsilon)}_m$. Let
\[ \vartheta_m:R^{(\epsilon)}_m\oplus R^{(-\epsilon)}_m\longepi W_m^{(\epsilon)}\oplus W_m^{(-\epsilon)}\longmono\Sigma_m^\vee \]
be the resulting map of $\tilde R_m$-modules; if $\bar x$ and $\bar y$ denote the images of $x$ and $y$ in $\Sigma_m^\vee$ then $\vartheta_m((\alpha,\beta))=\alpha\bar x+\beta\bar y$. Since there is a non-canonical isomorphism of $\cO_\p$-modules $R_m^{(\pm)}\simeq(\cO_\p/p^m\cO_\p)^{p^m}$, it turns out that $R_m^{(\pm)}$ is (non-canonically) isomorphic to $\bigl(R_m^{(\pm)}\bigr)^\vee$. Fix isomorphisms
\begin{equation} \label{i_m}
i_m^{(\pm)}:\bigl(R_m^{(\pm)}\bigr)^\vee\overset\simeq\longrightarrow R_m^{(\pm)}
\end{equation}
of $R_m$-modules and set $i_m:=i_m^{(\epsilon)}\oplus i_m^{(-\epsilon)}$. Composing the Pontryagin dual $\vartheta_m^\vee$ of $\vartheta_m$ with $i_m$, we get a map $i_m\circ\vartheta_m^\vee$ of $\tilde R_m$-modules that we still denote by  
\begin{equation} \label{theta_m-dual}
\vartheta_m^\vee:\Sigma_m\longrightarrow R_m^{(\epsilon)}\oplus R_m^{(-\epsilon)}. 
\end{equation}
If $\overline{\Sigma}_m:=\Sigma_m/\ker(\vartheta_m^\vee)$ then there is an injection $\bar\vartheta_m^\vee:\overline{\Sigma}_m\longmono  R_m^{(\epsilon)}\oplus R_m^{(-\epsilon)}$ of $\tilde R_m$-modules. Define $\overline{\Sigma}_m^{(\epsilon)}:=(\bar{\vartheta}_m^{\vee})^{-1}\bigl(R_m^{(\epsilon)}\oplus\{0\}\bigr)$ and $\overline{\Sigma}_m^{(-\epsilon)}:=(\bar{\vartheta}_m^{\vee})^{-1}\bigl(\{0\}\oplus R_m^{(-\epsilon)}\bigr)$. It follows that there is a splitting 
\begin{equation} \label{eq-Sigma_m}
\overline{\Sigma}_m=\overline{\Sigma}_m^{(\epsilon)}\oplus \overline{\Sigma}_m^{(-\epsilon)}
\end{equation}
of $\tilde R_m$-modules. Taking $G_m$-invariants, we obtain an injection 
\[ \bigl(\overline{\Sigma}_m^{(\pm)}\bigr)^{G_m}\longmono\bigl(R_m^{(\pm)}\bigr)^{G_m}\simeq\cO_\p/p^m\cO_\p \]
of $\cO_\p/p^m\cO_\p$-modules, and the structure theorem for finitely generated $\cO_\p$-modules implies that $\bigl(\overline{\Sigma}_m^{(\pm)}\bigr)^{G_m}\simeq\cO_\p/p^{m^{(\pm)}}\cO_\p$ for suitable integers $m^{(\pm)}\leq m$.

\subsection{Compatibility of the maps} \label{sec-duals} 

In order to ensure compatibility of the various maps appearing in the previous subsection as $m$ varies, in the sequel it will be useful to make a convenient choice of the isomorphism $i_m^{(\epsilon)}$ introduced in \eqref{i_m}.  
 
Let 
\[ \pi_m:H_f^1(K_m,A_{p^m})^\vee\longepi\mathcal E_m^\vee \] 
be the dual of the inclusion $\mathcal E_m\subset H_f^1(K_m,A_{p^m})$; since $H_f^1(K_m,A_{p^m})$ is discrete, the map $\pi_m$ is surjective. Since $y\in\ker(\pi)$, we have $\pi_m(Z_m)=0$, so there is a surjection $\bar\pi_m:\Sigma_m^\vee \twoheadrightarrow \mathcal E_m^\vee$ showing that $\mathcal E_m$ is actually a submodule of $\Sigma_m$. Since $(\bar\pi_m\circ \vartheta_m)\bigl(\{0\}\oplus R_m^{(-\epsilon)}\bigr)=\{0\}$, again because $y\in \ker(\pi)$, the dual of $\bar\pi_m\circ \vartheta_m$ factors through a map $\tilde\psi_m:\mathcal E_m\rightarrow\bigl(R_m^{(\epsilon)}\bigr)^\vee$. Using \eqref{equation9}, one easily checks the relation 
\[ u_m\beta_m=p\cdot\cores_{K_{m+1}/K_m}(\beta_{m+1}) \] 
for some $u_m\in R_m^\times$ (recall that we denote by $\beta_m$ the image of the Heegner cycle $\alpha_m$ in $\mathcal E_m$). Using this relation, one can prove the existence of isomorphisms $i_m^{(\epsilon)}:\bigl(R_m^{(\epsilon)}\bigr)^\vee\xrightarrow\simeq R_m^{(\epsilon)}$ such that if $\psi_m$ is the composition 
\[ \psi_m:\mathcal E_m\xrightarrow{\tilde\psi_m}\bigl(R_m^{(\epsilon)}\bigr)^\vee\xrightarrow{i_m^{(\epsilon)}}R_m^{(\epsilon)} \] 
then the cyclic $R_m$-modules $\psi_m(\mathcal E_m)$ are generated by elements $\theta_m\in R_m$ satisfying $\theta_\infty:={(\theta_m)}_m\in\Lambda$. From now on, fix $i_m^{(\epsilon)}$ as before, so that $\theta_\infty\in\Lambda$. In the following, we will identify $R_m^{(\epsilon)}$ and its Pontryagin dual by means of this map without making it explicit. We will also implicitly identify $R_m^{(-\epsilon)}$ with its Pontryagin dual, but we will not need to specify a convenient isomorphism in this case. 
 
\subsection{Galois extensions} \label{galois-subsec}

In this subsection we introduce several Galois extensions attached to the modules defined in \S \ref{sec-complex-conj}. We start with the following general discussion. 

For any $\mathcal O_\p/p^m\mathcal O_\p$-submodule $S\subset H^1_f(K_m,A_{p^m})$, we define the extension $M_S$ of $K_m(A_{p^m})$ cut out by $S$ as follows. Set 
\[ \mathcal G_m:=\Gal(K_m(A_{p^m})/K_m). \] 
With a slight abuse, we shall often view $\mathcal G_m$ as a subgroup of $\GL_2(\cO_\p/p^m\cO_\p)$, according to convenience. By \cite[Lemma 6.2]{BesDM}, the image of $\rho_m$ in $\GL_2(\cO_\p/p^m\cO_\p)$ contains a subgroup that is conjugate to $\GL_2(\Z/p^m\Z)$. Therefore
\begin{equation} \label{lemma-G_m} 
\text{$\mathcal G_m$ contains a subgroup that is conjugate to $\GL_2(\Z/p^m\Z)$ in $\GL_2(\cO_\p/p^m\cO_\p)$.}
\end{equation}

We need a variant of \cite[Proposition 6.3, (2)]{BesDM}, which we prove in the next

\begin{lemma} \label{vanishing-lemma}
$H^1\bigl(\Gal(K_m(A_{p^m})/K_m),A_{p^m}\bigr)=0$.
\end{lemma}

\begin{proof} Restriction of automorphisms injects $\Gal(K_m(A_{p^m})/K_m)$ into $\Gal(K(A_{p^m})/K)$; this induces a natural identification
\begin{equation} \label{gal-ident-eq} 
\Gal\bigl(K_m(A_{p^m})/K_m\bigr)=\Gal\bigl(K(A_{p^m})/K(A_{p^m})\cap K_m\bigr). 
\end{equation}
From the inflation-restriction exact sequence we can extract the exact sequence
\begin{equation} \label{transgression-eq}
\small{\begin{split}
   H^1\bigl(\Gal(K(A_{p^m})/K),A_{p^m}\bigr)&\longrightarrow H^1\bigl(\Gal(K(A_{p^m})/K(A_{p^m})\cap K_m),A_{p^m}\bigr)^{\Gal(K(A_{p^m})\cap K_m/K)}\\
   &\longrightarrow H^2\bigl(\Gal(K(A_{p^m})\cap K_m/K),A_{p^m}(K(A_{p^m})\cap K_m)\bigr),
\end{split}} 
\end{equation}
where the first map is restriction and the second is transgression. By \cite[Proposition 6.3, (2)]{BesDM}, $H^1\bigl(\Gal(K(A_{p^m})/K),A_{p^m}\bigr)=0$. On the other hand, the extension $K_m/\Q$ is solvable, hence $A_{p^m}(K(A_{p^m})\cap K_m)=0$ by \cite[Lemma 3.10, (2)]{LV}. As a consequence, the middle term in \eqref{transgression-eq} is trivial. But $H^1\bigl(\Gal(K(A_{p^m})/K(A_{p^m})\cap K_m),A_{p^m}\bigr)$ and $\Gal(K(A_{p^m})\cap K_m/K)$ are both $p$-groups, hence \cite[Lemma 3]{Serre} implies that
\[ H^1\bigl(\Gal(K(A_{p^m})/K(A_{p^m})\cap K_m),A_{p^m}\bigr)=0. \]
Now the lemma follows from \eqref{gal-ident-eq}. \end{proof}

Thanks to the vanishing result of Lemma \ref{vanishing-lemma}, restriction gives an injection 
\[ H^1_f(K_m,A_{p^m})\longmono H^1_f\bigl(K_m(A_{p^m}),A_{p^m}\bigr)^{\mathcal G_m}. \]
Define $G_{K_m(A_{p^m})}^{\rm ab}:=\Gal\bigl(K_m(A_{p^m})^{\rm ab}/K_m(A_{p^m})\bigr)$ where $K_m(A_{p^m})^{\rm ab}$ is the maximal abelian extension of $K_m(A_{p^m})$. It follows that there is an identification 
\[ H^1\bigl(K_m(A_{p^m}),A_{p^m}\bigr)^{\mathcal G_m}=\Hom_{\mathcal G_m}\bigl(G_{K_m(A_{p^m})}^{\rm ab},A_{p^m}\bigr) \] 
of $\cO_\p/p^m\cO_\p$-modules, where $\Hom_{\mathcal G_m}(\bullet,\star)$ stands for the group of $\mathcal G_m$-homomorphisms. Thus we obtain an injection of $\cO_\p/p^m\cO_\p$-modules 
\begin{equation} \label{eq13}
S\longmono\Hom_{\mathcal G_m}\bigl(G_{K_m(A_{p^m})}^{\rm ab},A_{p^m}\bigr),\qquad s\longmapsto\varphi_s,
\end{equation}
and for every $s\in S$ we let $M_s$ be the subfield of $K_m(A_{p^m})^{\rm ab}$ fixed by $\ker(\varphi_s)$. In other words, $M_s$ is the smallest abelian extension of 
${K_m(A_{p^m})}$ such that the restriction of $\varphi_s$ to $\Gal(K_m(A_{p^m})^{\rm ab}/M_s)$ is trivial. The maps $\varphi_s$ induce injections  
\[ \varphi_s:\Gal\bigl(M_s/K_m(A_{p^m})\bigr)\longmono A_{p^m} \]
of $\mathcal G_m$-modules. Let $M_S\subset K_m(A_{p^m})^{\rm ab}$ denote the composite of all the fields $M_s$ for $s\in S$.

Now we prove that the map 
\begin{equation} \label{eq14}
\Gal\bigl(M_S/K_m(A_{p^m})\bigr)\longrightarrow\Hom(S,A_{p^m}),\qquad g\longmapsto\bigl(s\mapsto\varphi_s({g|}_{M_s})\bigr) 
\end{equation}
is a $\mathcal G_m$-isomorphism; here $\Hom(\bullet,\star)$ is a shorthand for $\Hom_{\cO_\p/p^m\cO_\p}(\bullet,\star)$. Furthermore, we show that the map \eqref{eq13} induces an isomorphism 
\begin{equation} \label{eq17} 
S\overset\simeq\longrightarrow\Hom_{\mathcal G_m}\bigl(\Gal(M_S/K_m(A_{p^m})),A_{p^m}\bigr)
\end{equation}
of $\cO_\p/p^m\cO_\p$-modules. First of all, note that \eqref{eq13} gives an injection  
\begin{equation} \label{inj-S-eq}
S\longmono\Hom_{\mathcal G_m}\bigl(\Gal(M_S/K_m(A_{p^m})),A_{p^m}\bigr),\qquad s\longmapsto\bigl(g\mapsto\varphi_s({g|}_{M_s})\bigr) 
\end{equation}
of $\cO_\p/p^m\cO_\p$-modules.

Since $S$ is a finite $\cO_\p$-module of exponent $p^m$, there is an isomorphism 
\[ \xi:S\overset\simeq\longrightarrow\prod_{i=1}^t\cO_\p/p^{m_i}\cO_\p \] 
of $\cO_\p/p^m\cO_\p$-modules for suitable integers $m_i\leq m$. For every $i\in\{1,\dots,t\}$ set $S_i:=\xi^{-1}\bigl(\cO_\p/p^{m_i}\cO_\p\bigr)$ and choose a generator $s_i$ of the free $\cO_\p/p^{m_i}\cO_\p$-module $S_i$. The map $\varphi_{s_i}$ is killed by $p^{m_i}$, so it induces an injection
\[ \varphi_{s_i}:\Gal\bigl(M_{s_i}/K_m(A_{p^m})\bigr)\longmono A_{p^{m_i}}. \] 
On the other hand, for every $i$ there is a $\mathcal G_m$-equivariant injection 
\[ \Gal\bigl(M_{s_i}/K_m(A_{p^m})\bigr)\longmono\Hom(S_i,A_{p^m}),\qquad g\longmapsto\bigl(s_i\mapsto\varphi_{s_i}(g)\bigr). \]
But $\Hom(S_{s_i},A_{p^m})\simeq A_{p^{m_i}}$, hence there is a $\mathcal G_m$-equivariant injection
\begin{equation} \label{galois-inj-eq}
\Gal\bigl(M_S/K_m(A_{p^m})\bigr)\longmono\Hom(S,A_{p^m})\simeq\prod_{i=1}^tA_{p^{m_i}}. 
\end{equation}
Using the fact that the Galois representation $\rho_1$ on $A_p$ is irreducible (\cite[Lemma 3.8]{LV}), it can be shown that the image of \eqref{galois-inj-eq} is of the form $\prod_{i=1}^{t'}A_{p^{m_i'}}$ for integers $t'\leq t$ and $m_i'\leq m_i$, hence there is a $\mathcal G_m$-equivariant isomorphism $\Gal(M_S/K_m(A_{p^m}))\simeq \prod_{i=1}^{t'}A_{p^{m_i'}}$. 

\begin{proposition} \label{equiv-prop}
$\Hom_{\mathcal G_m}\bigl(A_{p^{m_i'}},A_{p^m}\bigr)\simeq\cO_\p/{p^{m_i'}}\cO_\p$.
\end{proposition}

\begin{proof} First of all, there is a canonical identification 
\[ \Hom_{\mathcal G_m}\bigl(A_{p^{m_i'}},A_{p^m}\bigr)=\Hom_{\mathcal G_m}\bigl(A_{p^{m_i'}},A_{p^{m'_i}}\bigr). \]
To ease the notation, set $n:=m'_i$. Since $A_{p^n}$ is free of rank $2$ over $\cO_\p/p^n\cO_\p$, there is an isomorphism $\Hom_{\cO_\p/p^n\cO_\p}(A_{p^n},A_{p^n})\simeq\M_2(\cO_\p/p^n\cO_\p)$; to obtain it, we fix a basis of $A_{p^n}$ over $\cO_\p/p^n\cO_\p$. Observe that the group $\mathcal G_m\subset\GL_2(\cO_\p/p^m\cO_\p)$ acts on $A_{p^n}$ via the reduction map $\GL_2(\cO_\p/p^m\cO_\p)\rightarrow\GL_2(\cO_\p/p^n\cO_\p)$. Let $\mathcal G_m^{(n)}$ be the image of $\mathcal G_m$ in $\GL_2(\cO_\p/p^n\cO_\p)$, so that
\begin{equation} \label{hom-eq} 
\Hom_{\mathcal G_m}(A_{p^n},A_{p^m})=\Hom_{\mathcal G_m^{(n)}}(A_{p^n},A_{p^n}). 
\end{equation}
By \eqref{lemma-G_m}, $\mathcal G_m$ contains a subgroup that is conjugate to $\GL_2(\Z/p^m\Z)$, hence $\mathcal G_m^{(n)}$ contains a subgroup that is conjugate to $\GL_2(\Z/p^n\Z)$. Write $g\GL_2(\Z/p^n\Z)g^{-1}$ with $g\in\GL_2(\cO_\p/p^n\cO_\p)$ for this subgroup. In particular, we want to find the matrices $A\in\M_2(\cO_\p/p^n\cO_\p)$ such that 
\[ (g X g^{-1})A(gX^{-1}g^{-1})=A \] 
for all $X\in\GL_2(\Z/p^n\Z)$. Upon setting $B:=g^{-1}Ag$, this is equivalent to describing the matrices $B\in\M_2(\cO_\p/p^n\cO_\p)$ such that 
\[ XBX^{-1}=B \] 
for all $X\in\GL_2(\Z/p^n\Z)$. Clearly, these are exactly the diagonal matrices, so we conclude that $\Hom_{\mathcal G_m^{(n)}}(A_{p^n},A_{p^n})\simeq\cO_\p/p^n\cO_\p$. The proposition follows from \eqref{hom-eq}. \end{proof} 

It follows from Proposition \ref{equiv-prop} that there is an isomorphism of $\cO_\p/p^m\cO_\p$-modules 
\[ \Hom_{\mathcal G_m}\bigl(\Gal(M_S/K_m(A_{p^m})),A_{p^m}\bigr)\simeq\prod_{i=1}^{t'}\cO_\p/{p^{m_i'}}\cO_\p. \]
By \eqref{inj-S-eq}, $S$ injects into $\Hom_{\mathcal G_m}\bigl(\Gal(M_S/K_m(A_{p^m})),A_{p^m}\bigr)$, and then comparing cardinalities gives $t=t'$ and $m_i=m_i'$ for all $i$, thus proving \eqref{eq14} and \eqref{eq17} simultaneously. From this we also deduce the following 

\begin{proposition} \label{S'-prop}
Given subgroups $S'\subset S\subset H^1_f(K_m,A_{p^m})$,  there is a canonical isomorphism of groups 
\[ \Gal(M_S/M_{S'})\simeq\Hom(S/S',A_{p^m}). \]
Conversely, for every subgroup $\bar S$ of $S/S'$ there is a subextension $M_{\bar S}/M_{S'}$ of $M_S/M_{S'}$ such that 
\[ \Gal(M_{\bar S}/M_{S'})\simeq\Hom(\bar S,A_{p^m}). \]
\end{proposition}

\begin{proof} The short exact sequence $0\rightarrow S'\rightarrow S\rightarrow S/S'\rightarrow0$ gives a short exact sequence
\[ 0\longrightarrow\Hom(S/S',A_{p^m})\longrightarrow\Hom(S,A_{p^m})\longrightarrow\Hom(S',A_{p^m})\longrightarrow0, \] 
where the surjectivity on the right comes from isomorphism \eqref{eq14} and the surjectivity of the canonical map from $\Gal(M_{S}/K_m(A_{p^m}))$ to $\Gal(M_{S'}/K_m(A_{p^m}))$. The first claim follows from the fact that, again by \eqref{eq14}, the kernel of the third arrow is naturally isomorphic to $\Gal(M_S/M_{S'})$. 

As for the second assertion, suppose that $\bar S$ is a subgroup of $S/S'$. Then, as before, one can define a subextension $M_{\bar S}/M_{S'}$ of $M_S/M_{S'}$; namely, an element $s\in\bar S$ cuts out an extension $M_s/M_{S'}$ and we let $M_{\bar S}$ be the composite of all the fields $M_s$. Replacing $S$ with $\bar S$, $K_m(A_{p^m})$ with $M_{S'}$ and $\mathcal G_m$ with $\Gal(M_{S'}/K_m)$ in the discussion above, one checks that there are isomorphisms of groups 
\[ \bar S\simeq\Hom_{\Gal(M_{S'}/K_m)}\bigl(\Gal(M_{\bar S}/M_{S'}),A_{p^m}\bigr),\qquad\Gal(M_{\bar S}/M_{S'})\simeq\Hom(\bar S,A_{p^m}), \]
as required. \end{proof}

With notation as in Proposition \ref{S'-prop}, we say that $M_S/M_{S'}$ is the extension associated with the quotient $S/S'$. Now we apply these constructions to the setting of \S \ref{sec-complex-conj}. 

\begin{lemma} \label{inj-selmer-lemma}
There is an injection
\[ H^1_f(K_m,A_{p^m})\longmono H^1_f(K_{m+1},A_{p^{m+1}}) \]
induced by the restriction map and the inclusion $A_{p^m}\hookrightarrow A_{p^{m+1}}$.
\end{lemma} 

\begin{proof} The extension $K_{m+1}/\Q$ is solvable, hence $A_{p^m}(K_{m+1})=0$ by \cite[Lemma 3.10, (2)]{LV}. It follows that restriction induces an injection
\begin{equation} \label{inj1-eq}
H^1(K_m,A_{p^m})\longmono H^1(K_{m+1},A_{p^m}).
\end{equation} 
Now consider the short exact sequence of $G_{K_{m+1}}$-modules
\begin{equation} \label{inj2-eq}
0\longrightarrow A_{p^m}\longrightarrow A_{p^{m+1}}\longrightarrow A_{p^{m+1}}/A_{p^m}\longrightarrow0. 
\end{equation}
Multiplication by $p^m$ yields a Galois-equivariant isomorphism $A_{p^{m+1}}/A_{p^m}\simeq A_p$, therefore $(A_{p^{m+1}}/A_{p^m})(K_{m+1})=0$ again by \cite[Lemma 3.10, (2)]{LV}. In light of this vanishing, passing to the long exact $G_{K_{m+1}}$-cohomology sequence associated with \eqref{inj2-eq} gives an injection
\begin{equation} \label{inj3-eq}
H^1(K_{m+1},A_{p^m})\longmono H^1(K_{m+1},A_{p^{m+1}}).
\end{equation}
Combining \eqref{inj1-eq} and \eqref{inj3-eq}, we get an injection
\[ H^1(K_m,A_{p^m})\longmono H^1(K_{m+1},A_{p^{m+1}}) \] 
that restricts to the desired injection between Selmer groups. \end{proof}

Thanks to Lemma \ref{inj-selmer-lemma}, in the sequel we shall often view $H^1_f(K_m,A_{p^m})$ as a subgroup of $H^1_f(K_{m+1},A_{p^{m+1}})$. Let $M_m$ denote the field cut out by the whole group $S=H^1_f(K_m,A_{p^m})$; then $M_m\subset M_{m+1}$. 

\begin{proposition} \label{surj-gal-prop}
There is a canonical surjection 
\[ \Gal\bigl(M_{m+1}/K_{m+1}(A_{p^{m+1}})\bigr)\longepi\Gal\bigl(M_m/K_m(A_{p^m})\bigr). \]
\end{proposition}

\begin{proof} View $H^1_f(K_m,A_{p^m})$ as a subgroup of $H^1_f(K_{m+1},A_{p^{m+1}})$ and consider the extension $H$ of $K_{m+1}(A_{p^{m+1}})$ cut out by it inside $M_{m+1}$. Then
\[ \Gal\bigl(H/K_{m+1}(A_{p^{m+1}})\bigr)\simeq\Hom\bigl(H^1_f(K_m,A_{p^m}),A_{p^{m+1}}\bigr)=\Hom\bigl(H^1_f(K_m,A_{p^m}),A_{p^m}\bigr), \] 
where the equality is a consequence of $H^1_f(K_m,A_{p^m})$ being $p^m$-torsion. On the other hand, by definition of $M_m$ there is an isomorphism $\Gal(M_m/K_m(A_{p^m}))\simeq\Hom(H^1_f(K_m,A_{p^m}),A_{p^m})$. Comparing, we find an isomorphism
\[ \Gal\bigl(H/K_{m+1}(A_{p^{m+1}})\bigr)\simeq\Gal\bigl(M_m/K_m(A_{p^m})\bigr) \]
that can be combined with the canonical surjection
\[ \Gal\bigl(M_{m+1}/K_{m+1}(A_{p^{m+1}})\bigr)\longepi\Gal\bigl(H/K_{m+1}(A_{p^{m+1}})\bigr) \]
to complete the proof. \end{proof}

Graphically, there is a diagram of field extensions 
\[ \xymatrix@R=20pt{M_{m+1}\ar@{-}[d]\ar@{-}[rdd]&\\
                    H\ar@{-}[d]&\\
                    K_{m+1}(A_{p^{m+1}})\ar@{-}[rd]&M_m\ar@{-}[d]\\
                    &K_m(A_{p^m})} \]
with $\Gal\bigl(H/K_{m+1}(A_{p^{m+1}})\bigr)\simeq\Gal\bigl(M_m/K_m(A_{p^m})\bigr)$. Let us define
\[ M_\infty:=\dirlim_m M_m,\qquad K_\infty(A):=\dirlim_m K_m(A_{p^m}), \]
so that 
\[ \Gal\bigl(M_\infty/K_\infty(A)\bigr):=\invlim_m\Gal\bigl(M_m/K_m(A_{p^m})\bigr), \]
the inverse limit being taken with respect to the maps of Proposition \ref{surj-gal-prop}. By \eqref{eq14}, for every $m\geq0$ there is an isomorphism $\Gal(M_m/K_m(A_{p^m}))\simeq\Hom(H^1_f(K_m,A_{p^m}),A_{p^m})$ of $\cO_\p/p^m\cO_\p[\mathcal G_m]$-modules, hence there is an isomorphism of $\cO_\p[\![\mathcal G_\infty]\!]$-modules  
\[ \Gal\bigl(M_\infty/K_\infty(A)\bigr)\simeq\Hom\bigl(H^1_f(K_\infty,A),A\bigr) \] 
where $\cO_\p[\![\mathcal G_\infty]\!]:=\sideset{}{_m}\invlim\cO_\p[\mathcal G_m]$ is defined with respect to the natural maps $\mathcal G_{m+1}\rightarrow\mathcal G_m$. 

Now recall the map $\vartheta_m^\vee$ of \eqref{theta_m-dual} and let $L_m\subset M_m$ be the extension of $K_m(A_{p^m})$ cut out by $\ker(\vartheta_m^\vee)$. Then there are canonical $\mathcal G_m$-isomorphisms
\[ \Gal\bigl(L_m/K_m(A_{p^m})\bigr)\simeq\Hom\bigl(\ker(\vartheta_m^\vee),A_{p^m}\bigr) \]  
and  
\begin{equation} \label{L_m-isom2-eq}
\Gal(M_m/L_m)\simeq\Hom\bigl(\overline{\Sigma}_m,A_{p^m}\bigr)\simeq\Hom\!\Big(\overline{\Sigma}_m^{(\epsilon)},A_{p^m}\Big)\oplus\Hom\!\Big(\overline{\Sigma}_m^{(-\epsilon)},A_{p^m}\Big);
\end{equation}
here \eqref{L_m-isom2-eq} is a consequence of Proposition \ref{S'-prop}. Let $L_m^{(\pm)}/L_m$ be the subextension of $M_m/L_m$ corresponding to $\overline{\Sigma}_m^{(\pm)}$ (cf. Proposition \ref{S'-prop}); then $L_m^{(\epsilon)}\cap L_m^{(-\epsilon)}=L_m$ and $M_m=L_m^{(\epsilon)}\cdot L_m^{(-\epsilon)}$. Finally, let $\tilde L_m^{(\pm)}$ be the extension of $L_m^{(\pm)}$ corresponding to $\bigl(\overline{\Sigma}_m^{(\pm)}\bigr)^{G_m}$. We have $\tilde L_m^{(+)}\cap\tilde L_m^{(-)}=L_m$ and  
\[ \Gal\bigl(\tilde L_m^{(\pm)}/L_m\bigr)\simeq\Hom\!\Big(\bigl(\overline{\Sigma}_m^{(\pm)}\bigr)^{G_m},A_{p^m}\Big). \]
Moreover, if $\tilde L_m:=\tilde L_m^{(+)}\cdot\tilde L_m^{(-)}$ then 
\[ \begin{split}
   \Gal\bigl(\tilde L_m/L_m\bigr)&\simeq\Hom\!\Big(\bigl(\overline{\Sigma}_m^{(+)}\bigr)^{G_m},A_{p^m}\Big)\oplus\Hom\!\Big(\bigl(\overline{\Sigma}_m^{(-)}\bigr)^{G_m},A_{p^m}\Big)\\&\simeq\Hom\!\Big(\bigl(\overline{\Sigma}_m\bigr)^{G_m},A_{p^m}\Big),
   \end{split} \]
where the second isomorphism follows by taking $G_m$-invariants in \eqref{eq-Sigma_m}. Since, by Lemma \ref{inj-selmer-lemma}, $H^1_f(K_m,A_{p^m})$ injects via restriction into $H^1_f(K_{m+1},A_{p^{m+1}})$, restriction induces an injection $\bigl(\overline{\Sigma}_m\bigr)^{G_m}\hookrightarrow\bigl(\overline{\Sigma}_{m+1}\bigr)^{G_{m+1}}$. It follows that for every $m\geq0$ there is a canonical projection 
\begin{equation} \label{cheb2}
\Gal\bigl(\tilde L_{m+1}/L_{m+1}\bigr)\longepi\Gal\bigl(\tilde L_m/L_m\bigr).
\end{equation} 

To introduce the last field extensions that we need, we dualize the exact sequence
\[ R_m^{(\epsilon)}\oplus R_m^{(-\epsilon)}\xrightarrow{\vartheta_m}\Sigma_m^\vee\longrightarrow\Sigma_m^\vee/\mathrm{im}(\vartheta_m)\longrightarrow0 \] 
and get an isomorphism $\ker(\vartheta_m^\vee)\simeq\bigl(\Sigma_m^\vee/\mathrm{im}(\vartheta_m)\bigr)^{\!\vee}$. Furthermore, dualizing the short exact sequence
\[ 0\longrightarrow\mathrm{im}(\vartheta_m)\longrightarrow\Sigma_m^\vee\longrightarrow\ker(\vartheta_m^\vee)^\vee\longrightarrow0 \] 
gives a short exact sequence
\begin{equation} \label{dual-seq-eq}
0\longrightarrow\ker(\vartheta_m^\vee)\longrightarrow\Sigma_m\xrightarrow{\vartheta_m^\vee}\mathrm{im}(\vartheta_m)^\vee\longrightarrow0.
\end{equation}
Finally, with maps $\vartheta$ and $p_m$ defined as in \eqref{theta-eq} and \eqref{p_m-eq}, let $U_m$ be the $\tilde R_m$-submodule of $H^1_f(K_m,A_{p^m})$ such that there is an indentification 
\begin{equation} \label{image-eq}
{\rm im}(p_m\circ\vartheta)=\Big(H_f^1(K_m,A_{p^m})/U_m\Big)^{\!\vee}. 
\end{equation}
Namely, set for simplicity $\mathcal I_m:={\rm im}(p_m\circ\vartheta)$ and consider the short exact sequence
\[ 0\longrightarrow\mathcal I_m\longrightarrow H_f^1(K_m,A_{p^m})^\vee\longrightarrow H_f^1(K_m,A_{p^m})^\vee/\mathcal I_m\longrightarrow0. \]
Since $\mathcal I_m$ is compact, hence closed in $H_f^1(K_m,A_{p^m})^\vee$, dualizing the sequence above gives a short exact sequence
\begin{equation} \label{duals-eq}
0\longrightarrow\Big(H_f^1(K_m,A_{p^m})^\vee\big/\mathcal I_m\Big)^{\!\vee}\longrightarrow H^1_f(K_m,A_{p^m})\longrightarrow\mathcal I_m^\vee\longrightarrow0. 
\end{equation}
Define $U_m:=\bigl(H_f^1(K_m,A_{p^m})^\vee/\mathcal I_m\bigr)^{\!\vee}$ and regard $U_m$ as an $\tilde R_m$-submodule of $H^1_f(K_m,A_{p^m})$ via \eqref{duals-eq}. Then there is a natural identification
\begin{equation} \label{quotient-eq}
H^1_f(K_m,A_{p^m})/U_m=\mathcal I_m^\vee,
\end{equation}
and dualizing \eqref{quotient-eq} gives \eqref{image-eq}.

Now write $\tilde M_m$ for the field cut out by $U_m$. As $p_m\circ\theta$ factors through $R_m^{(\epsilon)}\oplus R_m^{(-\epsilon)}$, there is a commutative diagram 
\[ \xymatrix{\Lambda^{(\epsilon)}\oplus\Lambda^{(-\epsilon)}\ar@{->>}[d]\ar[r]^-\vartheta&\mathcal X_\infty\ar[r]^-{p_m}&H^1_f(K_m,A_{p^m})^\vee\ar@{->>}[d]\\R_m^{(\epsilon)}\oplus R_m^{(-\epsilon)}\ar[rr]^-{\vartheta_m}&&\Sigma_m^\vee} \] 
which induces a surjection $\mathcal I_m\twoheadrightarrow\mathrm{im}(\vartheta_m)$ and then, by duality, an injection $\mathrm{im}(\vartheta_m)^\vee\hookrightarrow\mathcal I_m^\vee$. From this we obtain a commutative diagram with exact rows
\begin{equation} \label{d22}
\xymatrix{0\ar[r]& \ker(\vartheta_m^\vee)\ar[r]\ar@{^(->}[d]&\Sigma_m\ar[r]\ar@{^(->}[d]&\mathrm{im}(\vartheta_m)^\vee\ar[r]\ar@{^(->}[d]&0\\0\ar[r]&U_m\ar[r]& H^1_f(K_m,A_{p^m})\ar[r]&\mathcal I_m^\vee\ar[r]&0}
\end{equation} 
whose upper row is \eqref{dual-seq-eq}. Denote by $L_m^*$ the field corresponding to $\Sigma_m$, so that $\tilde L_m\subset L_m^*$ by Proposition \ref{S'-prop}. Observe that $\tilde M_m$ and $\tilde L_m$ are linearly disjoint over $L_m$. To check this, note that $\tilde M_m\cap\tilde L_m\subset\tilde M_m\cap L^*_m$ and that the second intersection corresponds to the subgroup $U_m\cap\Sigma_m$ inside $H^1_f(K_m,A_{p^m})$. But diagram \eqref{d22} shows that $\ker(\vartheta_m^\vee)=U_m\cap\Sigma_m$, hence $\tilde M_m\cap L^*_m=L_m$; we conclude that  $\tilde M_m\cap\tilde L_m=L_m$. It follows that
\[ \Gal\bigl(\tilde M_m\cdot\tilde L_m/\tilde M_m\bigr)\simeq\Gal\bigl(\tilde L_m/\tilde M_m\cap\tilde L_m\bigr)\simeq\Gal\bigl(\tilde L_m/L_m\bigr), \] 
and then the inclusion $\tilde M_m\cdot\tilde L_m\subset M_m$ induces a surjection
\begin{equation} \label{gal-onto-eq}
\Gal\bigl(M_m/\tilde M_m\bigr)\longepi\Gal\bigl(\tilde L_m/L_m\bigr). 
\end{equation}
It follows that for every $m$ there is a commutative square of surjective maps 
\begin{equation} \label{diagram}
\xymatrix{\Gal\bigl(M_{m+1}/\tilde M_{m+1}\bigr)\ar@{->>}[r]\ar@{->>}[d]&\Gal\bigl(\tilde L_{m+1}/L_{m+1}\bigr)\ar@{->>}[d]\\\Gal\bigl(M_m/\tilde M_m\bigr)\ar@{->>}[r]&\Gal\bigl(\tilde L_m/L_m\bigr)}
\end{equation} 
where the horizontal arrows are given by \eqref{gal-onto-eq} and the right vertical arrow is given by \eqref{cheb2}. The only things that remain to be checked are the surjectivity of the left vertical map and the commutativity of \eqref{diagram}. First of all, we want to show that $\tilde M_{m+1}\cap M_m=\tilde M_m$; this intersection corresponds to the subgroup $U_{m+1}\cap H^1_f(K_m,A_{p^m})$ inside $H^1_f(K_{m+1},A_{p^{m+1}})$, with $U_{m+1}$ defined as in \eqref{quotient-eq}. Set $\mathcal I_{m+1}:=\mathrm{im}(p_{m+1}\circ\vartheta)$; the injection between Selmer groups of Lemma \ref{inj-selmer-lemma} gives a surjection  $\mathcal I_{m+1}\twoheadrightarrow\mathcal I_m$ and then, by duality, an injection $\mathcal I_m^\vee\hookrightarrow\mathcal I_{m+1}^\vee$. Therefore there is a commutative diagram with exact rows
\[ \xymatrix{0\ar[r]&U_m\ar@{^(->}[d]\ar[r]&H^1_f(K_m,A_{p^m})\ar@{^(->}[d]\ar[r]&\mathcal I_m^\vee\ar@{^(->}[d]\ar[r]&0\\0\ar[r]&U_{m+1}\ar[r]&H^1_f(K_{m+1},A_{p^{m+1}})\ar[r]&\mathcal I_{m+1}^\vee\ar[r]&0} \]
where the injectivity of the map $U_m\rightarrow U_{m+1}$ is forced by the injectivity of the map between Selmer groups. We conclude that the subgroup $U_{m+1}\cap H^1_f(K_m,A_{p^m})$, which corresponds to $\tilde M_m$, is equal to the image of $U_m$ inside $H^1_f(K_{m+1},A_{p^{m+1}})$, and then $\tilde M_{m+1}\cap M_m=\tilde M_m$, as claimed above. It follows that
\[ \Gal\bigl(\tilde M_{m+1}\cdot M_m/\tilde M_{m+1}\bigr)\simeq\Gal\bigl(M_m/\tilde M_{m+1}\cap M_m\bigr)\simeq\Gal\bigl(M_m/\tilde M_m\bigr), \] 
and then the inclusion $\tilde M_{m+1}\cdot M_m\subset M_{m+1}$ yields the desired surjection
\[ \Gal\bigl(M_{m+1}/\tilde M_{m+1}\bigr)\longepi\Gal\bigl(M_m/\tilde M_m\bigr). \] 
This explains the existence and the surjectivity of the maps in diagram \eqref{diagram}. Finally, one checks that the commutativity of \eqref{diagram} is an immediate consequence of the constructions.  
%\[ \xymatrix@R=15pt{&M_{m+1}\ar@{-}[d]\ar@{-}[ddl]\ar@{-}[ddr]\\&\tilde M_{m+1}\ar@{-}[ddl]\ar@{-}[ddr]\\M_m\ar@{-}[rd]\ar@{-}[d]&&\tilde L_{m+1}\ar@{-}[dl]\ar@{-}[d]\\\tilde M_m\ar@{-}[rd]&\tilde L_m\ar@{-}[d]&L_{m+1}\ar@{-}[dl]\\&L_m} \]

\subsection{Families of Kolyvagin primes} \label{sec-families} 

The purpose of this subsection is to show that one can manufacture a compatible sequence $(\ell_m{)}_{m\geq1}$ of Kolyvagin primes, each $\ell_m$ being coprime with $p^m$; here the word ``compatibility'' refers to the canonical maps of Galois theory. More precisely, our goal is to prove the following 

\begin{proposition}\label{prop-cheb}  
There is a sequence $\ell_\infty=(\ell_m{)}_{m\geq1}$ of Kolyvagin primes satisfying the following conditions: 
\begin{enumerate}
\item $\Frob_{\ell_m}=[\tau g_m]$ in $\Gal(M_m/\Q)$ with $g_m\in\Gal\bigl(M_m/K_m(A_{p^m})\bigr)$ such that 
\[ (g_m{)}_m\in\Gal\bigl(M_\infty/K_\infty(A)\bigr); \] 
\item restriction induces an injective group homomorphism 
\[ \overline{\res}_{\ell_m}:\overline{\Sigma}_m\longmono H^1_f(K_{m,\ell_m},A_{p^m}); \]
\item $(\ell_m+1\pm a_{\ell_m})/p^m$ are invertible in $\cO_\p/p^m\cO_\p$. 
\end{enumerate}
\end{proposition}

\begin{proof} For each sign $\pm$ choose $h_m^{(\pm)}\in\Gal\bigl(\tilde L^{(\pm)}_m/L_m\bigr)$ such that the period of $\bigl(h^{(\pm)}_m\bigr)^\tau h_m^{(\pm)}$ is $p^{m^{(\pm)}}$. To justify the existence of an element with this property, observe that if $h_m^{(\pm)}$ corresponds to the homomorphism $\phi:\bigl(\overline{\Sigma}^{(\pm)}_m\bigr)^{G_m}\rightarrow A_{p^m}$ then
$\bigl(h^{(\pm)}_m\bigr)^\tau h_m^{(\pm)}$ corresponds to $x\mapsto \pm\tau\phi(x)+\phi(x)$. In light of this, to show the existence of such an $h_m^{(\pm)}$ it suffices to choose a $\phi$ that takes a generator of $\bigl(\overline{\Sigma}_m^{(\pm)}\bigr)^{G_m}$ to an element of order $p^{m^{(\pm)}}$ in $A_{p^m}^{(\pm)}$. Define $h_m:=\bigl(h_m^{(+)},h_m^{(-)}\bigr)\in\Gal\bigl(\tilde L_m/L_m\bigr)$ and choose the sequence $(h_m{)}_{m\geq1}$ so that the image of $h_{m+1}$ via surjection \eqref{cheb2} is $h_m$. Using diagram \eqref{diagram}, select also a compatible sequence of elements $g_m\in\Gal(M_m/K_m(A_{p^m}))$ such that the image of $g_m$ in $\Gal\bigl(\tilde L_m/L_m\bigr)$ is $h_m$. For every integer $m\geq1$ choose a prime number $\ell_m$ such that 
\begin{equation} \label{choice-of-primes}
\Frob_{\ell_m}=[\tau g_m]\quad\text{in $\Gal(M_m/\Q)$}.
\end{equation}
Clearly, $\ell_m$ is a Kolyvagin prime and the required compatibility conditions are fulfilled by construction, so (1) is satisfied. To check (2), we must show that the restriction is injective. For this, 
fix a prime $\mathfrak l_m$ of $M_m$ above $\ell_m$ satisfying $\Frob_{\mathfrak l_m/\ell_m}=\tau g_m$. Then the restriction of $\Frob_{\mathfrak l_m/\ell_m}$ to $\Gal(\tilde L_m/L_m)$ corresponds to an injective homomorphism 
\[ \phi_{\mathfrak l_m/\ell_m}:\bigl(\overline\Sigma_m\bigr)^{G_m}\longmono A_{p^m} \]
consisting in the evaluation at $\Frob_{\mathfrak l_m/\ell_m}$; namely, one has 
\[ \phi_{\mathfrak l_m/\ell_m}(s)=s\bigl(\Frob_{\mathfrak l_m/\ell_m}\bigr) \] 
for all $s\in\bigl(\overline\Sigma_m\bigr)^{G_m}$. The choice of $\mathfrak l_m$ determines a prime $\tilde \lambda_m$ of $K_m$ above $\ell_m$, and the completion of $K_m$ at $\tilde \lambda_m$ is canonically isomorphic to the completion $K_{\lambda_m}$ of $K$ at the unique prime $\lambda_m$ of $K$ above $\ell_m$. It follows that the canonical restriction map 
\begin{equation} \label{r}
\bigl(\overline\Sigma_m\bigr)^{G_m}\longmono H^1_f(K_{\lambda_m},A_{p^m})
\end{equation}
is injective, since the same is true of the composition of this map with the isomorphism 
\[ H^1_f(K_{\lambda_m},A_{p^m})\simeq A_{p^m} \]
that is given precisely by evaluation at Frobenius (recall \eqref{lemma-kol}). Suppose now that $s\in\overline{\Sigma}_m$ is non-zero and $\overline{\res}_{\ell_m}(s)=0$. In particular, the submodule $(R_ms)^{G_m}$ of $\bigl(\overline{\Sigma}_m\bigr)^{G_m}$ is sent to zero, via \eqref{r}, in the direct summand $H^1_f(K_{\lambda_m},A_{p^m})$ of $H^1_f(K_{m,\ell_m},A_{p^m})$ corresponding to $\tilde\lambda_m$. Up to multiplying $s$ by a suitable power of $p$, we may assume that $s$ is $p$-torsion. Now $R_ms$ is a non-trivial $\cO_\p/p\cO_\p$-vector space on which the $p$-group $G_m$ acts. By \cite[Proposition 26]{Serre}, the submodule $(R_ms)^{G_m}$ is non-trivial, and this contradicts the injectivity of \eqref{r}. Summing up, we have proved that all choices of a sequence $\ell_\infty={(\ell_m)}_{m\geq1}$ satisfying \eqref{choice-of-primes} enjoy properties (1) and (2) in the statement of the proposition. 

The finer choice of a sequence $\ell_\infty$ satisfying (3) as well can be made by arguing as in the proof of \cite[Proposition 12.2, (3)]{Nek}; see the proof of \cite[Proposition 3.26]{LV} for details. \end{proof}

\subsection{Local duality} \label{duality-sec}

The aim of this subsection is to bound the $\Lambda$-rank of $\Lambda x\oplus\Lambda y$ by a $\Lambda$-module $V(\ell_\infty)$ that surjects onto $\Lambda x\oplus\Lambda y$; the module $V(\ell_\infty)$ depends on the choice of a compatible family $\ell_\infty={(\ell_m)}_{m\geq1}$ of Kolyvagin primes as in \S \ref{sec-families}.  

By \cite[Proposition 3.8]{BK}, if $E$ is a number field and ${v}$ is a prime of $E$ then there is a perfect pairing 
\[ \langle\cdot,\cdot\rangle_{E,{v}}:H^1(E_{v},T/p^mT)\times H^1(E_{v},A_{p^m})\longrightarrow \Q_p/\Z_p \]
under which $H^1_f(E_{v},T/p^mT)$ and $H^1_f(E_{v},A_{p^m})$ are exact annihilators of each other. From this one deduces the existence (see, e.g., \cite[eq. (34)]{LV}) of a $\tau$-antiequivariant isomorphism   
\[ \delta_{E,{v}}:H^1_s(E_{v},T/p^mT)\overset\simeq\longrightarrow H^1_f(E_v,A_{p^m})^\vee. \]
Now let $\ell$ be a Kolyvagin prime and write $\delta_{m,\lambda}$ as a shorthand for $\delta_{K_m,\lambda}$. Taking the direct sum of the maps $\delta_{m,\lambda}$ over all the primes $\lambda\,|\,\ell$, we get a $\tau$-antiequivariant isomorphism 
\begin{equation} \label{delta-ell}
\delta_{m,\ell}:H^1_s(K_{m,\ell},T/p^mT)\overset{\simeq}\longrightarrow H^1_f(K_{m,\ell},A_{p^m})^\vee.
\end{equation}
Composing $\delta_{m,\ell}$ with the dual of the restriction $\res_{m,\ell}:H^1_f(K_m,A_{p^m})\rightarrow H^1_f(K_{m,\ell},A_{p^m})$, we get a map 
\[ H^1_s(K_{m,\ell},T/p^mT)\xrightarrow{\delta_{m,\ell}} H^1_f(K_{m,\ell},A_{p^m})^\vee\xrightarrow{\res_{m,\ell}^\vee} H^1_f(K_m,A_{p^m})^\vee \]
whose image we denote by $V_m(\ell)$. When $m$ is understood, we shall simply write 
$\delta_{\ell}$ for $\delta_{m,\ell}$ and $V(\ell)$ for $V_m(\ell)$. 

\begin{proposition} \label{lemma4.4}
Let $\ell_\infty$ be the sequence of Kolyvagin primes of Proposition \ref{prop-cheb}.
\begin{enumerate} 
\item For every $m\geq1$ there is a canonical surjection $V(\ell_m)\twoheadrightarrow W_m^{(\epsilon)}\oplus W_m^{(-\epsilon)}$. 
\item For every $m\geq1$ there is a canonical surjection $V(\ell_{m+1})\twoheadrightarrow V(\ell_m)$. 
\end{enumerate}
\end{proposition}

\begin{proof}  Fix an integer $m\geq1$. Composing the isomorphism $\delta_{m,\ell_m}$ in \eqref{delta-ell} with the dual of the map $\overline{\res}_{\ell_m}$ introduced in part (2) of Proposition \ref{prop-cheb}, we get a surjection 
\[ H^1_s(K_{m,\ell_m},T/p^mT)\xrightarrow{\delta_{\ell_m}}H^1_f(K_{m,\ell},A_{p^m})^\vee\xrightarrow{\overline{\res}_{\ell_m}^\vee}\overline{\Sigma}_m^\vee \] 
that, by definition, factors through $V(\ell_m)=V_m(\ell_m)$. Now $\overline{\Sigma}_m^\vee\simeq{\rm im}(\vartheta_m)$, which is isomorphic to $W_m^{(\epsilon)}\oplus W_m^{(-\epsilon)}$. 
Therefore we get an $\tilde R_m$-equivariant surjection 
\[ V(\ell_m)\longepi W_m^{(\epsilon)}\oplus W_m^{(-\epsilon)}, \]
which proves part (1). 

As for part (2), let us define the map $V_{m+1}(\ell_{m+1})\rightarrow V_m(\ell_m)$. There is a commutative diagram
\[ \xymatrix{H^1_s(K_{m+1},T/p^{m+1}T)\ar@{->>}[r]\ar^-{\cores}[d]&V_{m+1}(\ell_{m+1})\ar@{^(->}[r]& H^1_f(K_{m+1},A_{p^{m+1}})^\vee\ar^-{\res^\vee}[d]\\
H^1_s(K_{m},T/p^{m}T)\ar@{->>}[r]&V_m(\ell_m))\ar@{^(->}[r]& H^1_f(K_{m},A_{p^{m}})^\vee} \]
in which we have written simply ``res'' for the map $H^1_f(K_m,A_{p^m})\hookrightarrow H^1_f(K_{m+1},A_{p^{m+1}})$ of Lemma \ref{inj-selmer-lemma} and ``cores'' for the analogous map $H^1_s(K_{m+1},T/p^{m+1}T)\rightarrow H^1_s(K_m,T/p^mT)$ induced by corestriction. This shows that the dual of the restriction map gives a map $V_{m+1}(\ell_{m+1})\rightarrow V_m(\ell_m)$. Furthermore, by definition, ${\rm im}(\vartheta_m)$ is a quotient of ${\rm im}(\vartheta_{m+1})$, hence the map $\overline{\Sigma}_{m+1}^\vee\twoheadrightarrow\overline{\Sigma}_m^\vee$ is surjective. Now there is yet another commutative diagram 
\[ \xymatrix{R_{m+1}^{(+)}\oplus R_{m+1}^{(-)}\simeq H^1_s(K_{m+1,\ell_{m+1}},T/p^{m+1}T)\ar@{->>}[d]\ar@{->>}[r]&
V_{m+1}(\ell_{m+1})\ar@{->>}[r]\ar[d]&\overline{\Sigma}_{m+1}^\vee\ar@{->>}[d] \\
R_{m}^{(+)}\oplus R_{m}^{(-)}\simeq H^1_s(K_{m,\ell_m},T/p^{m}T)\ar@{->>}[r]&V_m(\ell_m)\ar@{->>}[r]&\overline{\Sigma}_m^\vee} \] 
where the vertical maps are (co)restrictions or their duals and the leftmost vertical arrow is surjective because the same is true of the map $R_{m+1}\rightarrow R_m$ (here we are implicitly using \eqref{lemma-kol} and the identification provided by local duality). The desired surjectivity follows. \end{proof}

In the rest of the paper, let $\ell_\infty={(\ell_m)}_{m\geq1}$ be the sequence of Kolyvagin primes constructed in Proposition \ref{prop-cheb}. It follows from part (2) of Proposition \ref{lemma4.4} that we can define the $\Lambda$-module 
\[ V(\ell_\infty):=\invlim_mV(\ell_m). \] 

\begin{proposition} \label{prop4.5}
There is a surjection 
\[ V(\ell_\infty)\longepi\Lambda x\oplus\Lambda y. \]
\end{proposition}

\begin{proof} The inverse limit of the maps in part (1) of Proposition \ref{lemma4.4} gives a surjection 
\begin{equation} \label{eq-surj}
V(\ell_\infty)\longepi\invlim_m\bigl(W_m^{(\epsilon)}\oplus W_m^{(-\epsilon)}\bigr),
\end{equation} 
where we use the fact that the projective system satisfies the Mittag--Leffler condition, as all the modules involved are finite. On the other hand, there is a short exact sequence 
\[ 0\longrightarrow Z_m\longrightarrow {\rm im}(p_m\circ\vartheta)\longrightarrow W_m^{(\epsilon)}\oplus W_m^{(-\epsilon)}\longrightarrow 0, \] 
and passing to inverse limits shows that there is a short exact sequence 
\[ 0\longrightarrow\invlim_mZ_m\longrightarrow\Lambda x\oplus \Lambda y\longrightarrow\invlim_m\bigl(W_m^{(\epsilon)}\oplus W_m^{(-\epsilon)}\bigr)\longrightarrow0. \]
Since $\Lambda x\cap\Lambda y=\{0\}$ and $\sideset{}{_m}\invlim Z_m\subset\Lambda x\cap\Lambda y$, we have $\sideset{}{_m}\invlim Z_m=0$, therefore $\Lambda x\oplus \Lambda y$ is isomorphic to $\sideset{}{_m}\invlim\bigl(W_m^{(\epsilon)}\oplus W_m^{(-\epsilon)}\bigr)$. Combining this with \eqref{eq-surj} gives the result. \end{proof}

\subsection{Kolyvagin classes} \label{kolyvagin-subsec}

We briefly review the construction of Kolyvagin classes attached to Heegner cycles. 

Let $\ell$ be a Kolyvagin prime relative to $p^m$; in particular, $p^m\,|\,\ell+1$ and $p^m\,|\,a_\ell$. Assume also that $p^{m+1}\nmid \ell+1\pm a_\ell$. Let $H_\ell$ be the ring class field of $K$ of conductor $\ell$. The fields $K_m$ and $H_\ell$ are linearly disjoint over the Hilbert class field $H_1$ of $K$, and $\Gal(H_\ell/H_1)$ is cyclic of order $\ell+1$. Let $H_{m,\ell}^{(p)}$ be the maximal subextension of the composite $K_mH_\ell $ having $p$-power degree over $K_m$ and let 
\[ \alpha(\ell):=\cor_{H_{\ell p^{m+1}}/H_{m,\ell}^{(p)}}(y_{\ell p^{m+1}})\in H^1_f\bigl(H_{m,\ell}^{(p)},T\bigr) \] 
be the corestriction from $H_{\ell p^{m+1}}$ to $H_{m,\ell}^{(p)}$ of the Heegner cycle $y_{\ell p^{m+1}}\in H^1_f(H_{\ell p^{m+1}},T)$. 

Set $\mathfrak G_\ell:=\Gal\bigl(H_{m,\ell}^{(p)}/K_m\bigr)$. By class field theory, if $n_\ell:=\ord_p(\ell+1)$ then $\mathfrak G_\ell\simeq\Z/p^{n_\ell}\Z$; in particular, $m\,|\,n_\ell$. Fix a generator 
$\sigma_\ell$ of $\mathfrak G_\ell$ and consider the Kolyvagin operator 
\[ {\bf D}_\ell:=\sum_{i=1}^{p^{n_\ell}-1}i\sigma_\ell^i\in\Z/p^m\Z[\mathfrak G_\ell]. \] 
One has $(\sigma_\ell-1){\bf D}_\ell=-\tr_{H_{m,\ell}^{(p)}/K_m}$, and therefore 
\[ {\bf D}_\ell\bigl(\alpha(\ell)\bigr)\in H^1_f\bigl(H_{m,\ell}^{(p)},T/{p^m}T{\bigr)}^{\!\mathfrak G_\ell}. \]
Moreover, restriction induces an isomorphism 
\begin{equation} \label{res-isom-kol-eq}
\res_{H_{m,\ell}^{(p)}/K_m}:H^1_f(K_m,T/{p^m}T)\overset\simeq\longrightarrow H^1_f\bigl(H_{m,\ell}^{(p)},T/{p^m}T{\bigr)}^{\!\mathfrak G_\ell}. 
\end{equation}

We can give the following

\begin{definition}
The \emph{Kolyvagin class} $d(\ell)\in H^1(K_m,T/p^mT)$ attached to $\ell$ is the class corresponding to ${\bf D}_\ell(\alpha(\ell))$ under isomorphism \eqref{res-isom-kol-eq}.
\end{definition}

If $\ell$ is a Kolyvagin prime relative to $p^m$ then there exists (\cite[p.116]{Nek}) a $\tau$-antiequivariant isomorphism of $\cO_\p/p^m\cO_\p$-modules 
\[ \phi_{m,\ell}:H^1_s(K_{m,\ell},T/p^mT)\overset\simeq\longrightarrow H^1_f(K_{m,\ell},T/p^mT). \] 
Finally, recall the map $\partial_\ell$ that was defined at the end of \S \ref{sec-Bloch-Kato}, which consists of the sum of the projections to the singular local cohomology groups for all primes of $K_m$ dividing $\ell$. 

\begin{proposition} \label{prop-d-ell} 
The class $d(\ell)$ enjoys the following properties: 
\begin{enumerate}
\item if $v$ is a prime of $K_m$ not dividing $\ell$ then $\res_v(d(\ell))$ belongs to $H^1_f(K_m,T/p^mT)$; 
\item $\phi_{m,\ell}\bigl(\partial_\ell(d(\ell))\bigr)=u_\ell\cdot\res_\ell(\alpha_m)$ with $u_\ell\in\mathcal O_\p^\times$.  
\end{enumerate}
\end{proposition}

\begin{proof} This follows from \cite[Proposition 10.2]{Nek} for primes $v\nmid p$, while for primes $v\,|\,p$ it is a consequence of the de Rham conjecture for open varieties proved by Faltings. For details, see \cite[Proposition 3.17 and Lemma 3.19]{LV}. \end{proof} 

\subsection{Global duality} \label{duality-sec2}

In this subsection we use Kolyvagin classes, combined with global reciprocity laws, to bound the rank of the $\Lambda$-module $V(\ell_\infty)$, which was introduced in \S \ref{duality-sec} and surjects onto $\Lambda x\oplus\Lambda y$. 

Recall that if $E$ is a number field, $s\in H^1(E,A_{p^m})$ and $t\in H^1(E,T/p^mT)$ then
 \begin{equation} \label{global-duality} 
\sum_{v}{\big\langle\res_v(s),\res_v(t)\big\rangle}_{E,v}=0
\end{equation}
where $v$ ranges over all finite places of $E$ and ${\langle\cdot,\cdot\rangle}_{E,v}$ is the local Tate pairing at $v$. 

\begin{proposition} \label{prop-global-duality} 
The $\Lambda$-rank of $V(\ell_\infty)$ is at most $1$. 
\end{proposition}

\begin{proof} For every $m\geq1$ consider the Kolyvagin class $d(\ell_m)\in H^1(K_m,T/p^mT)$. Since there is an isomorphism of $\tilde R_m$-modules $\res_{\ell_m}(\mathcal E_m)\simeq R_m^{(\epsilon)}\theta_m$, where $\theta_m\in R_m$ is defined as in \S \ref{sec-duals}, part (2) of Proposition \ref{prop-d-ell} implies that 
\[ \res_{\ell_m}\bigl(R_m d(\ell_m)\bigr)\simeq R_m^{(-\epsilon)}\theta_m \] 
as $\tilde R_m$-modules. If we let $\xi_m^{(\pm)}$ be generators of $H^1(K_{m,\ell_m},T/p^mT)^{(\pm)}$ as $R_m^{(\pm)}$-module then 
\begin{equation} \label{eq24}
\res_{\ell_m}\bigl(R_md(\ell_m)\bigr)\cap\Big(H^1(K_{m,\lambda_m},T/p^mT)^{(\epsilon)}\oplus\{0\}\Big)=\{0\}.
\end{equation}
Using \cite[\S 1.2, Lemma 7]{Ber1}, write $\res_{\ell_m}(d(\ell_m))=\bigl(\rho_m\theta_m\xi_m^{(\epsilon)},\nu_m\theta_m\xi^{(-\epsilon)}_m\bigr)$ for suitable $\rho_m\in R_m$ and $\nu_m\in R_m^\times$. Define $W_m:=R_m\bigl(\rho_m\xi_m^{(\epsilon)},\nu_m\xi_m^{(-\epsilon)}\bigr)$. Then \eqref{eq24} gives a decomposition 
\[ H^1(K_{m,\ell_m},T/p^mT)\simeq H^1(K_{m,\ell_m},T/p^mT)^{(\epsilon)}\oplus W_m, \] 
from which we deduce that 
\[ \theta_mH^1(K_{m,\ell_m},T/p^mT)\simeq\theta_mH^1(K_{m,\ell_m},T/p^mT)^{(\epsilon)}\oplus \res_{\ell_m}\bigl(R_m d(\ell_m)\bigr). \]
By part (1) of Proposition \ref{prop-d-ell}, the class $d(\ell_m)$ is trivial at all the primes not dividing $\ell_m$, and then \eqref{global-duality} implies that $(\delta_{\ell_m}\circ\res_{\ell_m})(d(\ell_m))=0$. Summing up, we get 
\[ \theta_mV(\ell_m)\simeq\delta_{\ell_m}\bigl(\theta_m H^1(K_{m,\ell_m},T/p^mT)\bigr)\simeq\delta_{\ell_m}\bigl(\theta_m H^1(K_{m,\ell_m},T/p^mT)^{(\epsilon)}\bigr). \] 
It follows that $\theta_mV(\ell_m)$ is a cyclic $R_m$-module for all $m$. Since $\theta_\infty\in\Lambda$ and $\Lambda=\sideset{}{_m}\invlim R_m$, we conclude that $\theta_\infty V(\ell_\infty)$ is a cyclic $\Lambda$-module, and then $V(\ell_\infty)$ is cyclic over $\Lambda$ as well. \end{proof}

\subsection{Completion of the proof of the main result} \label{completion-subsec}

Recall from the beginning of Section \ref{euler-sec} that our goal is to show that the element $y\in\ker(\pi)$ is $\Lambda$-torsion. But this is immediate: since $\Lambda x$ is free of rank $1$, combining Propositions \ref{prop4.5} and \ref{prop-global-duality} shows that $\Lambda y$ is $\Lambda$-torsion, which concludes the proof. 

Finally, we record the following consequence of Theorem \ref{prop3.5}.

\begin{corollary} 
The $\Lambda$-module $V(\ell_\infty)$ has rank $1$.
\end{corollary}

\bibliographystyle{amsplain}
\bibliography{Iwasawa}
\end{document}